%% file: BPRoughDraft_V2.tex
\documentclass[11pt]{amsart}

\usepackage{graphicx,amsmath,amssymb,amsthm,bbm,url}
\usepackage{algorithm,algorithmic,caption,subcaption}
\usepackage{fullpage}
\usepackage{color,arydshln,verbatim}
\usepackage[dvipsnames]{xcolor}

\newtheorem{lem}{Lemma}
\newtheorem{cor}{Corollary}
\newtheorem{thm}{Theorem}

\DeclareMathOperator{\supp}{supp}

\def \a {\mathbf a}

\def \x {\mathbf x}

\def \y {\mathbf y}

\def \v {\mathbf v}
\def \X { X}

\def \m {\mathbf m}
\def \n {\mathbf n}
\newcommand{\R}{\ensuremath{\mathbb{R}}}  
\renewcommand{\H}{\ensuremath{\mathcal{H}}}
\newcommand{\tX}{\ensuremath{\widetilde{X}}}

\newcommand{\tx}{\ensuremath{\widetilde{\x}}}
\newcommand{\C}{\ensuremath{\mathbb{C}}}
\newcommand{\N}{\ensuremath{\mathbb{N}}}

\newcommand{\bigfrac}[2]{\ensuremath{\frac{\displaystyle #1}{\displaystyle #2}}}
\newcommand{\Span}{\ensuremath{\mathrm{span}}}

\newcommand{\ee}{\ensuremath{\mathbbm{e}}}
\newcommand{\ii}{\ensuremath{\mathbbm{i}}}
\newcommand{\diag}{\ensuremath{\mathrm{diag}}}
\newcommand{\sgn}{\ensuremath{\mathrm{sgn}}}
\newcommand{\vol}{\ensuremath{\mathrm{vol}}}

\begin{document}
\pagestyle{plain}

\title{Phase Retrieval from Local Measurements:  Improved Robustness via Eigenvector-Based Angular Synchronization}\thanks{Mark A. Iwen:  Department of Mathematics, and Department of Computational Mathematics, Science, and Engineering (CMSE), Michigan State University, East Lansing, MI, 48824, USA (markiwen@math.msu.edu).  Supported in part by NSF DMS-1416752. \\ \indent B. Preskitt:  Department of Mathematics, University of California San Diego, La Jolla, CA 92093, USA (bpreskitt@ucsd.edu)\\ \indent  R. Saab:  Department of Mathematics, University of California San Diego, La Jolla, CA 92093, USA (rsaab@ucsd.edu).  Supported in part by a Hellman Fellowship and the NSF under DMS-1517204 \\ \indent A. Viswanathan:  Department of Mathematics, Michigan State University, East Lansing, MI, 48824, USA (aditya@math.msu.edu).}
\author{Mark A. Iwen, Brian Preskitt, Rayan Saab, Aditya Viswanathan}

\begin{abstract}
\input{abstract.tex}
\end{abstract}

\maketitle
\thispagestyle{empty}

\section{Introduction}
\label{sec:intro}
\input{intro_V2.tex}
\section{The Spectrum of $\tilde{\X}_0$}
\label{sec:Spectrum}
\input{Spectrum_V2.tex}

\section{Perturbation Theory for $\tilde{\X}_0$}
\label{sec:Perturb}
\input{AltPerturb_V3.tex}

We are now properly equipped to analyze the robustness of Algorithm~\ref{alg:phaseRetrieval1} to noise.

\section{Recovery Guarantees for the Proposed Method}
\label{sec:RecovGuarantee}
\input{RecovGuarantee_V2.tex}

\section{Numerical Evaluation}
\label{sec:NumEval}
\input{NumEval.tex}
\section{Concluding Remarks}
\label{sec:conclusion}

In this paper new and improved deterministic robust recovery guarantees are proven for the phase retrieval problem using local correlation measurements.  In addition, a new practical phase retrieval algorithm is presented which is both faster and more noise robust than previously existing approaches (e.g., alternating projections) for such local measurements.  

Future work might include the exploration of more general classes of measurements which are guaranteed to lead to well conditioned linear systems  of the type used to reconstruct $X \approx X_0$ in line 1 of Algorithm~\ref{alg:phaseRetrieval1}.  Currently two deterministic measurement constructions are known (recall, e.g., Section~\ref{sec:MeasMatrix}) -- it should certainly be possible to construct more general families of such measurements.

Other interesting avenues of inquiry include the theoretical analysis of the magnitude estimate approach proposed in Section~\ref{sec:MagEstImpNumerical} in combination with the rest of Algorithm~\ref{alg:phaseRetrieval1}.  Alternate phase retrieval approaches might also be developed by using such local block eigenvector-based methods for estimating phases too, instead of just using the single global top eigenvector as currently done in line 3 of Algorithm~\ref{alg:phaseRetrieval1}.

Finally, more specific analysis of the performance of the proposed methods using masked/windowed Fourier measurements (recall Section~\ref{sec:STFT}) would also be interesting.  In particular, an analysis of the performance of such approaches as a function of the bandwidth of the measurement mask/window could be particularly enlightening.    

\bibliographystyle{abbrv}
\bibliography{SparsePR}

\section*{Appendix}
\input{Perturb.tex}

\section*{Acknowledgements}

The authors would like to thank Felix Krahmer for helpful discussions regarding Lemma~\ref{cor:rank1Bound}.
MI and RS would like to thank the Hausdorff Institute of Mathematics, Bonn for its hospitality during its Mathematics  of Signal Processing Trimester Program. A portion of this work was completed during that time. 

\end{document}

%% file: abstract.tex
%
We improve a phase retrieval approach that uses correlation-based measurements with compactly supported measurement masks \cite{IVW2015_FastPhase}.  The improved algorithm admits deterministic measurement constructions together with a robust, fast recovery algorithm that consists of solving a system of linear equations in a lifted space, followed by finding an eigenvector (e.g.,~via an inverse power iteration).  Theoretical reconstruction error guarantees from \cite{IVW2015_FastPhase} are improved as a result for the new and more robust reconstruction approach proposed herein.  Numerical experiments demonstrate robustness and computational efficiency that outperforms competing approaches on large problems.  Finally, we show that this approach also trivially extends to phase retrieval problems based on windowed Fourier measurements.

%% file: intro_V2.tex
\definecolor{Fgreen}{RGB}{34,139,34}




Consider the problem of recovering a vector $\x_0 \in \C^d$ from measurements $\y \in \R^D$ with entries $y_j$ given by
\begin{equation}
y_j = |\langle \mathbf{a}_j, \x_0 \rangle|^2 + \eta_j, \quad\quad j = 1,\ldots, D.\label{eq:phase_ret}
\end{equation}
Here the measurement vectors $\mathbf{a}_j\in \C^d$ are known and the scalars $\eta_j \in \R$ denote noise terms.  This problem is known as the \emph{phase retrieval problem} (see, e.g., \cite{walther1963question, millane1990phase}), as we may think of the $|\cdot|^2$ in \eqref{eq:phase_ret} as erasing the phases of the measurements $\langle \a_j, \x_0 \rangle$ in an otherwise linear system of equations.

The phase retrieval problem  arises in many important signal acquisition schemes, including crystallography and ptychography (e.g., \cite{millane1990phase
}), diffraction imaging \cite{gerchberg1972practical}, and optics \cite{millane1990phase, walther1963question}, 
  among many others.  Due to the breadth and importance of the applications, there has been significant interest in developing efficient algorithms to solve this problem. Indeed, one of the first algorithms proposed came in the early 1970's with the work of Gerchberg and Saxton \cite{gerchberg1972practical}. Since then many variations of their method have been proposed (e.g, \cite{
fienup1978reconstruction}) and used widely in practice.  On the other hand -- until recently -- there have not been theoretical guarantees concerning the conditions under which these algorithms recover the underlying signal and the extent to which they can tolerate measurement error. Nevertheless, starting in 2006 a growing body of work (e.g., \cite{ balan2007fast, balan2006signal, bandeira2013near, bodmann2013stable, candes2012phaselift, eldar2012phase, IVW2015_FastPhase, li2012sparse}) has emerged, proposing new methods with theoretical performance guarantees under various assumptions on the signal $\x_0$ and the measurement vectors $\a_j$. Unfortunately, the assumptions (especially on the measurement vectors) often do not correspond to the setups used in practice. In particular, the mathematical analysis often requires that the measurement vectors be random or generic (e.g., \cite{balan2006signal, bandeira2013near, candes2012phaselift}) while in practice the measurement vectors are a deterministic aspect of the imaging apparatuses employed.  A main contribution of this paper is analyzing a construction that more closely matches practicable and deterministic measurement schemes. We propose a two-stage algorithm for solving the phase retrieval problem in this setting and we analyze our method, providing upper bounds on the associated reconstruction error.

\subsection{Local Correlation Measurements}
\label{sec:locCorrMeas}

Consider the case where the vectors $\a_j$ represent shifts of compactly-supported vectors $\m_j, j = 1, \ldots, K$ for some $K \in \N$.  Using the notation $[n]_k:=\{k,\ldots,k+n-1\}\subset \N$, and defining $[n]:=[n]_1$ we take $\x_0, \m_j \in \C^d$ with $\supp(\m_j)=[\delta]\subset [d]$ for some $\delta \in \N$.  We also denote the space of Hermitian matrices in $\C^{k \times k}$ by $\H^k$.  Now we have measurements of the form \begin{equation} (\y_\ell)_j = |\langle \x_0, S^*_\ell \m_j \rangle|^2, \quad (j, \ell) \in [K] \times P, \label{eq:shift_model} \end{equation} where $P \subset [d]_0$ is arbitrary and $S_\ell : \C^d \to \C^d$ is the discrete circular shift operator, namely \begin{equation*} (S_\ell \x_0)_j = (\x_0)_{\ell + j}. \end{equation*}  One can see that \eqref{eq:shift_model} represents the modulus squared of the correlation between $\x_0$ and locally supported measurement vectors.  Therefore, we refer to the entries of $\y$ as local correlation measurements.
Following \cite{candes2013phaselift, IVW2015_FastPhase, balan2006signal}, the problem may be lifted to a linear system on the space of $\C^{d \times d}$ matrices.  In particular, we observe that 
\begin{align*}
	(\y_\ell)_j &= |\langle S_\ell \x_0, \m_j\rangle|^2 = \m_j^* (S_\ell \x_0) (S_\ell \x_0)^* \m_j \\
	&= \langle \x_0 \x_0^*, S_\ell^* \m_j \m_j^* S_\ell\rangle,
\end{align*}
where the inner product above is the Hilbert-Schmidt inner product. Restricting to the case $P = [d]_0$,  for every matrix $A \in \Span\{S_\ell^* \m_j \m_j^* S_\ell\}_{\ell, j}$ we have $A_{ij} = 0$ whenever $|i - j| \mod d \ge \delta$.  Therefore, we introduce the family of operators $T_k : \C^{d \times d} \to \C^{d \times d}$ given by \[T_k(A)_{ij} = \left\{\begin{array}{r@{,\qquad}l} A_{ij} & |i - j| \mod d < k \\ 0 & \text{otherwise}. \end{array}\right.\]  Note that $T_\delta$ is simply the orthogonal projection operator onto its range $T_\delta(\C^{d \times d}) \supseteq \Span\{S_\ell^* \m_j \m_j^* S_\ell\}_{\ell, j}$; therefore, \begin{equation} (\y_\ell)_j = \langle \x_0 \x_0^*, S_\ell^* \m_j \m_j^* S_\ell \rangle = \langle T_{\delta}(\x_0 \x_0^*), S_\ell^* \m_j \m_j^* S_\ell \rangle, \quad (j, \ell) \in [K] \times P. \label{eq:lifted_system} \end{equation}  

For convenience, we set $D := K|P|$ and define the map $\mathcal{A} : \C^{d \times d} \to \C^D$ 
\begin{equation}\mathcal{A}(X) = [\langle X, S_\ell^* \m_j \m_j^* S_\ell \rangle]_{(\ell, j)}.\label{eq:linear}\end{equation}  
Sometimes, we consider $\mathcal{A}|_{T_{\delta}(\C^{d \times d})}$, the restriction of $\mathcal{A}$ to the domain $T_{\delta}(\C^{d \times d})$; indeed, if this linear system is injective on $T_\delta(\C^{d \times d})$, then we can readily solve for 
\begin{equation}T_\delta(\x_0 \x_0^*) =: X_0 \label{equdef:X0} \end{equation}  
using our measurements $(\y_\ell)_j = (\mathcal{A}(\x_0 \x_0^*))_{(\ell,j)}$.
In \cite{IVW2015_FastPhase}, deterministic masks $\m_j$ were constructed for which \eqref{eq:lifted_system} was indeed invertible for certain choices of $K$ and $P$.  An additional construction is given below in \S\ref{sec:MeasMatrix}.

Improving on \cite{IVW2015_FastPhase}, we can further see that $\x_0$ can be deduced from $X_0$ up to a global phase in the noiseless case as follows:  First, $X_0$ immediately gives the magnitudes of the entries of $\x_0$ since $(X_0)_{ii} = |(x_0)_i|^2$.  
The only challenge remaining, therefore, is to find $\arg((x_0)_i)$ up to a global phase.  We proceed by defining $\tilde{\x}_0$ and $\widetilde{X}_0$ by \[(\tilde{x}_0)_i = \sgn((x_0)_i)\] \[(\widetilde{X}_0)_{ij} = \left\{\begin{array}{r@{,\quad}l} \sgn((X_0)_{ij}) & |i - j| \mod d < \delta \\ 0 & \text{otherwise} \end{array}\right.,\] where $\sgn : \C \to \C$ is the usual normalization mapping \[\sgn(z) = \left\{\begin{array}{r@{,\qquad}l} \dfrac{z}{|z|} & z \neq 0 \\ 1 & \text{otherwise} \end{array}\right..\]  Indeed, in \cite{IV_SPIE}, it was shown that the phases of the entries of $\x_0$ (up to a global phase) are given by the leading eigenvector of $\widetilde X_0$.  Moreover, it was shown that this leading eigenvector is unique.  Lemma \ref{lem:EigGap} of this paper improves in these results by giving a lower bound on the gap between the top two eigenvalues of $\widetilde X_0$.  This better understanding of the spectrum of $\widetilde X_0$ is then leveraged to analyze the robustness of this eigenvector-based phase retrieval method to measurement noise.




\subsection{Contributions}
\label{sec:mainRes}

In this paper, we analyze a phase retrieval algorithm (Algorithm \ref{alg:phaseRetrieval1}) for estimating a vector $\x_0$ from noisy localized measurements of the form
 \begin{equation} (\y_\ell)_j = |\langle \x_0, S^*_\ell \m_j \rangle|^2 + n_{j \ell}, \quad (j, \ell) \in [2\delta-1] \times [d]_0 \label{eq:shift_model_noise}.\end{equation}
 This algorithm is composed of two main stages. First, we apply the inverse of the linear operator
$$\mathcal{A}|_{T_{\delta}(\C^{d \times d})}: T_{\delta}(\C^{d\times d}) \to \C^{(2\delta-1)d}$$
  defined immediately after \eqref{eq:linear}, to obtain a Hermitian estimate ${X}$ of $T_\delta{(\x_0\x_0^*)}$ given by 
  \begin{equation}X = \Big( (\mathcal{A}|_{T_\delta(\C^{d\times d})})^{-1} {\bf y}\Big)/2 + \Big( (\mathcal{A}|_{T_\delta(\C^{d\times d})})^{-1} {\bf y}\Big)^*/2  \in T_\delta(\C^{d\times d}) \label{eq:X}
  .\end{equation} In particular, our choice of $\m_j$ as described in Section \ref{sec:MeasMatrix} ensures that $\mathcal{A}|_{T_{\delta}(\C^{d \times d})}$ is both invertible and well conditioned. Next, once we have an approximation of $T_\delta(\x_0 \x_0^*)$, we estimate the magnitudes and phases of the entries of $\x_0$ separately. 
  
  For the magnitudes, we simply use the square-roots of the diagonal entries of $X$. For the phases, we use the normalized eigenvector corresponding to the top eigenvalue of \begin{equation}\widetilde{X}:=\frac{X}{|X|},\label{eq:Xtilde}\end{equation}
where the operations are considered elements.  The hope is that the leading eigenvector of $\tX$ still serves as a good approximation to the leading eigenvector of \begin{equation}\tX_0:= \frac{X_0}{|X_0|}:= \frac{T_\delta(\x_0\x_0^*)}{|T_\delta(\x_0\x_0^*)|},\label{eq:X_0}\end{equation}
 which is seen in Section \ref{sec:Spectrum} (see also \cite{IV_SPIE}) to indeed be a scaled version of the phase vector \begin{equation}\tilde{\x}_0 := \frac{\x_0}{|\x_0|}\label{eq:tx_0}\end{equation} (up to a global phase ambiguity).  The entire method is summarized in Algorithm \ref{alg:phaseRetrieval1}, and its associated recovery guarantees are presented in Theorem \ref{Prop:RecovRes}, while its computational complexity is discussed after the theorem.
\begin{algorithm}
\renewcommand{\algorithmicrequire}{\textbf{Input:}}
\renewcommand{\algorithmicensure}{\textbf{Output:}}
\caption{Fast Phase Retrieval from Local Correlation Measurements}
\label{alg:phaseRetrieval1}
\begin{algorithmic}[1]
    \REQUIRE Measurements $\y\in \R^D$ as per \eqref{eq:shift_model_noise}
    \ENSURE ${\bf x} \in \mathbbm{C}^d$ with ${\bf x} \approx \mathbbm{e}^{-\mathbbm{i} \theta} {\bf x}_0$ for some $\theta \in [0, 2 \pi]$ 
    \STATE Compute the Hermitian matrix $X = \Big( (\mathcal{A}|_{T_\delta(\C^{d\times d})})^{-1} {\bf y}\Big)/2 + \Big( (\mathcal{A}|_{T_\delta(\C^{d\times d})})^{-1} {\bf y}\Big)^*/2  \in T_\delta(\C^{d\times d})$ as an estimate of $T_\delta(\x_0\x_0^*)$
    \STATE Form the banded matrix of phases, $\tilde{\X} \in T_\delta(\mathbbm{C}^{d \times d})$, by normalizing the non-zero entries of $\X$ 
    \STATE Compute the normalized top eigenvector of $\tilde{\X}$, denoted $\tilde{\x} \in \mathbbm{C}^d$, with $\| \tilde{\x} \|_2 = \sqrt{d}$
    \STATE Set $x_j = \sqrt{X_{j,j}} \cdot (\tilde{x})_j$ for all $j \in [d]$ to form $\x \in \mathbbm{C}^d$
    \end{algorithmic}
\end{algorithm}

\begin{thm}
Let $(x_0)_{\rm min} := \min_j |(x_0)_j|$ be the smallest magnitude of any entry in $\x_0 \in \mathbbm{C}^d$.  Then, the estimate $\x$ produced in Algorithm~\ref{alg:phaseRetrieval1} satisfies 
\[ \min_{\theta \in [0, 2 \pi]} \left\Vert  \x_0 - \mathbbm{e}^{\mathbbm{i} \theta} \x \right\Vert_2 \leq C \left( \frac{\Vert \x_0 
        \Vert_{\infty}}{(x_0)^2_{\rm min}} \right) \left( \frac{d}{\delta} \right)^2 \kappa \| \n \|_2 + C d^{\frac{1}{4}} \sqrt{\kappa \| \n \|_2 },\]
where $\kappa > 0$ is the condition number of the system \eqref{eq:X} and $C \in \mathbb{R}^+$ is an absolute universal constant.
\label{Prop:RecovRes}
\end{thm}

Theorem~\ref{Prop:RecovRes}, which deterministically depends on both the masks and the signal, is a strict improvement over the first deterministic theoretical robust recovery guarantees proven in \cite{IVW2015_FastPhase} for a wide class of non-vanishing signals.

%

Consider the computational complexity of Algorithm~\ref{alg:phaseRetrieval1} (assuming, of course, that $\mathcal{A}|_{T_\delta(\C^{d\times d})}$ is actually invertible).  One can see that line 1 can  always be done in at most $\mathcal{O}(d \cdot \delta^3 + \delta \cdot d \log d)$ flops using a block circulant matrix factorization approach (see Section 3.1 in \cite{IVW2015_FastPhase}). In certain cases one can improve on this; for example, the second (new) mask construction of Section \ref{sec:MeasMatrix} allows line 1 to be performed in only $\mathcal{O}(d\cdot\delta)$ flops.  Even in the worst case, however, if one precomputes this block circulant matrix factorization in advance given the masks $\m_j$ then line 1 can always be done in $\mathcal{O}(d \cdot \delta^2 + \delta \cdot d \log d)$ flops thereafter.  

The top eigenvector $\tilde{\x}$ of $\tilde{X}$ is guaranteed to be found in line 3 of Algorithm~\ref{alg:phaseRetrieval1} in the low-noise (e.g., noiseless) setting via the shifted inverse power method with shift $\mu := 2 \delta - 1$ and initial vector ${\bf e}_1$ (the first standard basis vector).  
More generally, one may utilize the Rayleigh quotient iteration with the initial eigenvalue estimate fixed to $2 \delta - 1$ for the first few iterations.
In either case, each iteration can be accomplished with $\mathcal{O}(d \cdot \delta^2)$ flops due to the banded structure of $\tilde{X}$ (see, e.g., \cite{trefethen1997numerical}).  In the low-noise setting the top eigenvector $\tilde{\x}$ can be computed to machine precision in $\mathcal{O}(\log d)$ such iterations,\footnote{To see why $\mathcal{O}(\log d)$ iterations suffice one can appeal to lemmas~\ref{lem:spectrum}~and~\ref{lem:EigGap} below.  Let $|\lambda_1| > |\lambda_2| \geq \dots \geq |\lambda_d|$ be the eigenvalues of $\tilde{X}$ with associated orthonormal eigenvectors ${\bf u}_j \in \mathbbm{C}^d$.  Let $\delta := |\lambda_1| - |\lambda_2| > 0$.  When the noise level is sufficiently low (so that $\tilde{X} \approx \tilde{X}_0$) one will have both $(i)$ $|{\bf e}^*_1 {\bf u}_j| = \Theta(1/\sqrt{d})~\forall j \in [d]$, and $(ii)$ $\mu \in (\lambda_1-\delta/4, \lambda_1+\delta/4)$ be true.  Thus, we will have that there exists some unit norm ${\bf r} \in \mathbbm{C}^d$ such that
$$\frac{\left( \tilde{X} - \mu I \right)^{-k}{\bf e}_1}{\left\| \left( \tilde{X} - \mu I \right)^{-k}{\bf e}_1 \right\|_2} = \frac{{\bf u}_1 + \sum^d_{j=2} \mathcal{O}\left( \left|\frac{\lambda_1 - \mu}{\lambda_j - \mu} \right|^k \right) {\bf u}_j}{1 + \mathcal{O}\left( \frac{d}{9^k} \right)} = {\bf u}_1 + \mathcal{O} \left( \frac{d}{3^k} \right) {\bf r}$$ holds for any given integer $k = \Omega \left( \log_3 d \right)$.} 
for a total flop count of $\mathcal{O}(\delta^2 \cdot d \log d)$ for line 3 in that case.  In total, then, one can see that Algorithm~\ref{alg:phaseRetrieval1} will always require just $\mathcal{O}(\delta^2 \cdot d \log d + d \cdot \delta^3)$ total flops in low-noise settings.  Furthermore, in all such settings a measurement mask support of size $\delta = \mathcal{O}(\log d)$ appears to suffice.

\subsection{Connection to Ptychography}

\begin{figure}[hbtp]
    \includegraphics[scale=0.25]{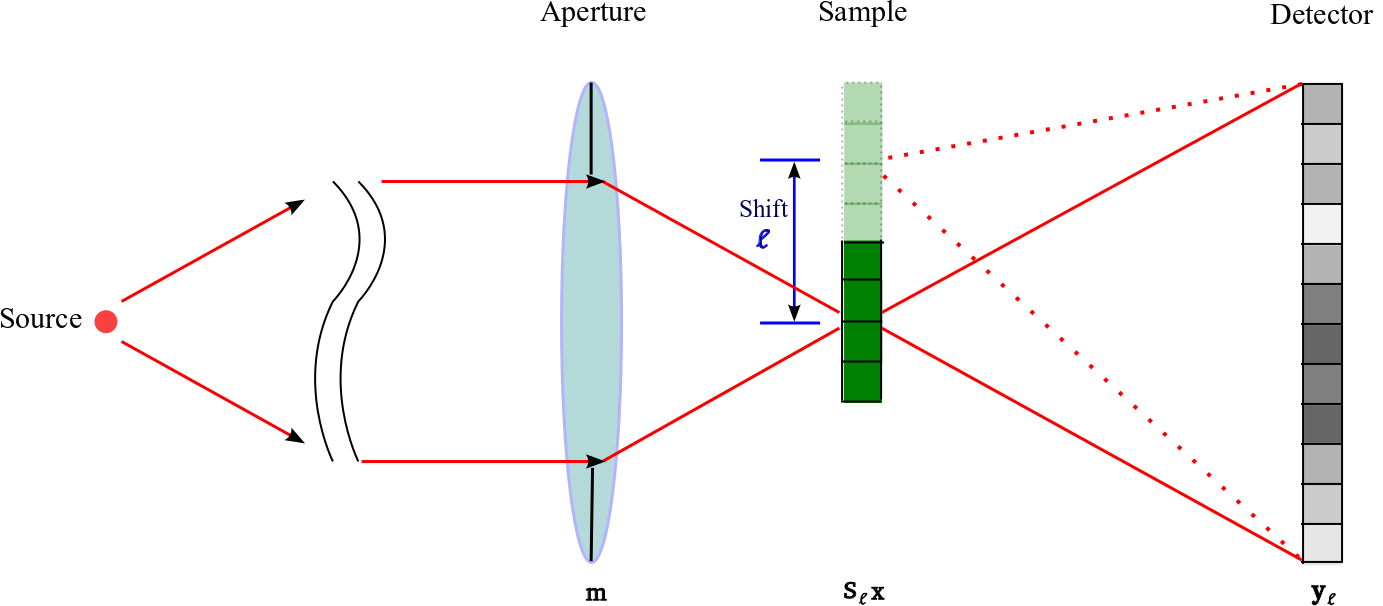}
    \caption{Illustration of one-dimensional ptychographic imaging {\small 
        (Adapted from ``Fly-scan ptychography'', Huang et al., 
        Scientific Reports 5 (9074), 2015.)} }
    \label{fig:ptychography_setup}
\end{figure}

In ptychographic imaging (see Fig.
\ref{fig:ptychography_setup}), small regions of a specimen are
illuminated one at a time and an intensity\footnote{By intensity,
we mean magnitude squared.} detector captures each of the
resulting  diffraction patterns. Thus each of the ptychographic
measurements is a local measurement, which under certain
assumptions (e.g., appropriate wavelength of incident radiation,
far-field Fraunhofer approximation), can be modeled as
\cite{Goodman2005IntroFourierOptics,Dierolf2008Ptych}
\begin{equation}
    y(t, \omega) = \left \vert \mathcal F [ \widetilde h \cdot S_t f ] (\omega) 
       \right \vert^2 + \eta(t, \omega) .\label{eq:ptych}
\end{equation} 
%
%
Here, $\mathcal F$ denotes the Fourier transform, $f : [0, 1] \to \C$ represents
the unknown test specimen, $S_t$ is the shift operator defined
via $$(S_t f)(s) := f(s + t),$$ and $\widetilde h : [0, 1] \to \C$ is the so-called
illumination function 
\cite{arXiv:1105.5628_MarchesiniEtAl}
of the imaging system. 
To account for the local nature of the measurements in
\eqref{eq:ptych}, we assume that $\text{supp}(\widetilde h)
\subset \text{supp} (f)$.

As the phase retrieval problem is inherently non-linear and requires sophisticated computer algorithms to solve, consider the discrete version of \eqref{eq:ptych}, with $\widetilde{\m}, \x_0 \in \C^d$ discretizing $\widetilde{h}$ and $f$.  Thus \eqref{eq:ptych}, in the absence of noise, becomes 
\begin{equation}
    (\y_\ell)_j = \left \vert \sum_{n=1}^d \widetilde m_n \, (x_0)_{n+\ell} \,
      \mathbbm e^{-\frac{2\pi \mathbbm i (j-1) (n-1)}d} 
      \right \vert^2, \quad (j,\ell) \in [d] \times  
      [d]_0, 
  \label{eq:1d_ptycho}
\end{equation}
where indexing is
considered modulo-$d$, so $(\y_\ell)_j$
is a diffraction measurement corresponding to the $j^{th}$ Fourier mode
of a circular $\ell$-shift of the specimen. We use circular shifts for convenience and we remark that this is appropriate as one can zero-pad $\x_0$ and $\widetilde{\m}$  in \eqref{eq:1d_ptycho} and obtain the same $(\y_\ell)_j$ as one would with non-circular shifts. In practice, one may not need to use all the shifts $\ell \in [d]_0$ as a subset may suffice.  Defining $\m_j \in \C^d$ by \begin{equation} (\m_j)_n = \overline{\widetilde m_n} \, \mathbbm e^{\frac{2\pi \mathbbm i (j-1) (n-1)}d} \label{eq:ptychm} \end{equation} and rearranging \eqref{eq:1d_ptycho}, we obtain
\begin{align}
    (\y_\ell)_j &= \left \vert \sum_{n=1}^d (x_0)_{n+\ell} \, \overline{(\m_j)_n} \right \vert^2 \label{eq:loco_measurements} 
        = \left \vert \sum_{n=1}^\delta (x_0)_{n+\ell} \, \overline{(\m_j)_n} \right \vert^2 \\
        &= \lvert \langle S_\ell \x_0, \m_j \rangle \rvert^2 \notag 
       = \langle S_\ell \x_0 \x_0^* S_\ell^*, \m_j \m_j^* \rangle \notag\\
        &= \langle T_\delta(\x_0 \x_0^*), S_\ell^* \m_j \m_j^* S_\ell \rangle, \quad (j, \ell) \in [d] \times [d]_0 \notag
\end{align}
where the second and last equalities follow from the fact that $\widetilde \m$ (and hence each $\m_j$) is locally supported.  
We note that \eqref{eq:loco_measurements}
defines a correlation with local masks or window functions
$\m_j$.  More importantly, \eqref{eq:loco_measurements} shows that ptychography (with $\ell$ ranging over any subset of $[d]_0$) represents a case of the general system seen in \eqref{eq:lifted_system}.

\subsection{Connections to Masked Fourier Measurements}
\label{sec:STFT}
%
%
Often, in imaging applications involving phase retrieval, a mask is placed either between the illumination source and the sample or between the sample and the sensor. Here, we will see that the mathematical setup that we consider is applicable in this scenario, albeit when the masks are band-limited.  As before, let $\x_0, \m \in \mathbbm C^d$ denote the unknown signal of
interest, and a known mask (or window), respectively. Moreover, for a vector $\x_0 \in \C^d$ we denote its discrete Fourier transform $\widehat{\x_0} \in \mathbbm{C}^d$ by $$(\widehat{x_0})_k :=\sum_{n=1}^d (x_0)_n e^{-2\pi \mathbbm{i} (n-1)(k-1)/d}.$$  Here, we consider squared magnitude \emph{windowed Fourier
transform} measurements of the form 
\begin{equation}
  ({\y_\ell})_k = \left \vert \sum_{n=1}^d (x_0)_n \, m_{n-\ell} \,
      \mathbbm e^{-\frac{2\pi \mathbbm i (k-1) (n-1)}d} 
      \right \vert^2, ~ k \in [d],
      ~ \ell \in \{\ell_1, \dots, \ell_L\} \subset 
        [d]_0.
  \label{eq:STFT_measurements}
\end{equation}
%
As before, $\ell$ denotes a shift or translation of the mask/window, so
 $({\y_\ell})_k$ corresponds to the (squared magnitude of) the
$k^{th}$ Fourier mode associated with an
$\ell$-shift\footnote{As above, all indexing and shifts are considered
modulo-$d$.} of the mask $\m$. Defining the modulation operator, $W_k: \C^d \mapsto \C^d$,  by its action $(W_k \x_0)_n = e^{2\pi \mathbbm i (k-1)(n-1)/d}~(x_0)_n$ and applying elementary Fourier transform properties\footnote{$\widehat{S_\ell \x_0}=W_{\ell+1}\widehat{\x_0}$, $\widehat{W_k \x_0}=S_{-k+1}\widehat{\x_0}$, and $W_k S_\ell \x_0 = e^{-2\pi \mathbbm i (k-1)\ell/d} S_\ell W_k \x_0$.}
 one has
\begin{align}
    (\y_\ell)_k &=  \lvert \langle  \x_0, S_{-\ell}(e^{2\pi \mathbbm i (k-1)\ell/d}W_k\overline{\m} )\rangle \rvert^2 \notag \\
       & =  \lvert \langle  \x_0, S_{-\ell}(W_k \overline{\m} )\rangle \rvert^2 \notag 
       =  \lvert \langle  \widehat{\x_0}, \widehat{S_{-\ell}(W_{k} \overline{\m} )}\rangle \rvert^2 \notag \\
       & = \lvert \langle  \widehat{\x_0}, W_{-\ell+1}(S_{-k+1}\widehat{\overline{\m}} )\rangle \rvert^2 \notag \\ 
       &=\lvert \langle  \widehat{\x_0}, S_{-k+1}(W_{-\ell+1}\widehat{\overline{\m}} )\rangle \rvert^2.
       \end{align}

Defining $\widehat{\m}_\ell := W_{-\ell+1}\widehat{\overline{\m}}$ and assuming that $\supp(\widehat{\overline{\m}})\subset [\delta]$ (e.g., assuming that $\m$ is real-valued and band-limited), we now have that
\begin{align}
    (\y_\ell)_{k}   &= \langle \widehat\x_0 \widehat\x_0^*, S_{-k+1} \widehat{\m}_\ell \widehat{\m}_\ell^* S_{-k+1}^*\rangle \notag\\
   &= \langle T_\delta(\widehat\x_0 \widehat\x_0^*), S_{-k+1} \widehat{\m}_\ell \widehat{\m}_\ell^* S_{-k+1}^*\rangle, \notag
\end{align}
which again represents a case of the general system seen in \eqref{eq:lifted_system}. Moreover, our results all hold for this setting, albeit with the Fourier transforms of signals and conjugated masks.

\subsection{Related Work}
\label{sec:lit}



The first approaches to the phase retrieval problem were proposed in the 1970's in \cite{gerchberg1972practical} by Gerchberg and Saxton, and were famously improved in \cite{fienup1978reconstruction} later that decade.  Though these techniques work well in practice and have been popular for decades, they are notoriously difficult to analyze.  These iterative methods work by improving an initial guess until they stagnate.  Recently Marchesini et al.~proved that alternating projection schemes using generic measurements are guaranteed to converge to the correct solution {\em if provided with a sufficiently accurate initial guess} and algorithms for ptychography were explored in particular \cite{marchesini2015alternating}.  However, no global recovery guarantees currently exist for alternating projection techniques using local measurements (i.e., finding a sufficiently accurate initial guess is not generally easy).

Other authors have taken to proving probabilistic recovery guarantees when provided with globally supported Gaussian measurements.  Methods for which such results exist vary in their approach, and include convex relaxations \cite{candes2014solving,candes2013phaselift}, gradient descent strategies \cite{candes2015phase}, graph-theoretic  \cite{alexeev2014phase} and frame-based approaches \cite{balan2009painless, bodmann2013stable}, and variants on the alternating minimization (e.g., with resampling) \cite{netrapalli2013phase}.

Several recovery algorithms achieve theoretical recovery guarantees while using at most $D = \mathcal{O}(d \log^4 d)$ masked Fourier coded diffraction pattern measurements, including both {\em PhaseLift} \cite{Candes2014WF,gross2015improved}, and {\em Wirtinger Flow} \cite{candes2015phase}.  However, these measurements are both randomized (which is crucial to the probabilistic recovery guarantees developed for both PhaseLift and Wirtinger Flow -- deterministic recovery guarantees do not exist for either method in the noisy setting), and provide global information about $\x_0$ from each measurement (i.e., the measurements are not locally supported).

Among the first treatments of local measurements are \cite{eldar2014sparse,BendoryE16} and \cite{jaganathan2015stft}, in which it is shown that STFT measurements with specific properties can allow (sparse) phase retrieval in the noiseless setting, and several recovery methods are proposed.  Similarly, the phase retrieval approach from \cite{alexeev2014phase} was extended to STFT measurements in \cite{salanevich2015polarization} in order to produce recovery guarantees in the noiseless setting.  More recently, randomized robustness guarantees were developed for time-frequency measurements in \cite{pfander2016robust}.  However, no {\it deterministic} robust recovery guarantees have been proven in the noisy setting for any of these approaches.  Furthermore, none of the algorithms developed in these papers are empirically demonstrated to be competitive numerically with standard alternating projection techniques for large signals when utilizing windowed Fourier and/or correlation-based measurements.  In \cite{IVW2015_FastPhase}, the authors propose the measurement scheme developed in the current paper and prove the first deterministic robustness results for a different greedy recovery algorithm.





\subsection{Organization}  Section \ref{sec:MeasMatrix} discusses two collections of local correlation masks $\m_j$, one of which is novel and the other of which was originally studied in \cite{IVW2015_FastPhase}.  Most importantly, Section \ref{sec:MeasMatrix} shows that the recovery of $T_\delta(\x_0\x_0^*)$ from measurements associated with the proposed masks can be done stably in the presence of measurement noise. 
Moreover, since in the noisy regime, the leading eigenvector $\widetilde{\x}$ of $\widetilde{X}$ (associated with line 3 of Algorithm \ref{alg:phaseRetrieval1}) will no longer correspond exactly to the true phases $\widetilde{\x}_0$, we are interested in a perturbation theory for the eigenvectors of $\widetilde{X}_0$.  Intuitively, $\widetilde{\x}$ will be most accurate when the eigenvalue of $\widetilde X_0$ associated with $\widetilde{\x}_0$ is well separated from the rest of the eigenvalues and so, accordingly, Section~\ref{sec:Spectrum} studies the spectrum of $\widetilde X_0$.  Indeed, this eigenvalue is rigorously shown to control the stability of the top eigenvector of $\widetilde{X}_0$ with respect to noise, and Section~\ref{sec:Perturb} develops perturbation results concerning their top eigenvectors by adapting the spectral graph techniques used in \cite{alexeev2014phase}.  Recovery guarantees for the proposed phase retrieval method are then compiled in Section~\ref{sec:RecovGuarantee}.  Numerical results demonstrating the accuracy, efficiency, and robustness of the proposed methods are finally provided in Section \ref{sec:NumEval}, while Section \ref{sec:conclusion} contains some concluding remarks and avenues for further research. In the appendix, we provide an alternate, weaker but easier to derive eigenvector perturbation result analogous to the one in Section \ref{sec:Perturb} which may be of independent interest.

\section{Well-conditioned measurement maps}
\label{sec:MeasMatrix}
Here, we present two example constructions for which the linear operator $\mathcal{A}|_{T_\delta(\C^{d\times d})}$ used in Step 1 of Algorithm \ref{alg:phaseRetrieval1} is well conditioned. Such constructions are crucial for the stability of the method to additive noise. 

\subsection*{Example 1:} 
In \cite{IVW2015_FastPhase}, a construction was proposed for the masks $\m_\ell$ in \eqref{eq:shift_model} that guarantees the stable invertibility of $\mathcal{A}$.  This construction comprises windowed Fourier measurements with parameters $\delta \in \mathbbm{Z}^+$ and $a \in [4, \infty)$ corresponding to the $2\delta-1$ masks $\m_j\in\C^d$, $j=1,...,2\delta-1$ with entries given by
\begin{equation}
(\m_j)_{n} = \left\{ \begin{array}{ll}
    \frac{\mathbbm{e}^{-n/a}}{\sqrt[4]{2\delta -1}} \cdot
    \mathbbm{e}^{\frac{2 \pi \mathbbm{i} \cdot (n - 1) \cdot (j -1)}
    {2\delta - 1}} & \textrm{if}~n \leq \delta \\ 
    0 & \textrm{if}~n >  \delta\end{array} \right. .
    \label{eq:MeasDef}
\end{equation}
Here, measurements using all shifts $\ell = 1,..., d$ of each mask are taken.  In the notation of \eqref{eq:lifted_system}, this corresponds to  $K = 2\delta - 1$ and $P = [d]_0$, which yields $D = (2\delta - 1)d$ total measurements.  By considering the basis $\{E_{ij}\}$ for $T_\delta(\C^{d \times d})$ given by 
\begin{align}\nonumber
E_{i,j}(s,t) =\left\{\begin{array}{l} 1, \quad (i,j)=(s,t)\\ 0, \quad \text{otherwise}\end{array}\right.
\end{align}
it was shown in \cite{IVW2015_FastPhase} that this system is both well conditioned and rapidly invertible.  In particular, if $M'$ is the matrix representing the measurement mapping $\mathcal{A} : T_{\delta}(\C^{d \times d}) \to T_{\delta}(\C^{d \times d})$ with respect to the basis $\{E_{ij}\}$, the following holds.

\begin{thm}[\cite{IVW2015_FastPhase}]
Consider measurements of the form \eqref{eq:MeasDef} with $a:= \max \left\{ 4,~\frac{\delta - 1}{2} \right\}$. 
Let $M' \in \mathbbm{C}^{D \times D}$ be the matrix representing the measurement mapping $\mathcal{A} : T_{\delta}(\C^{d \times d}) \to T_{\delta}(\C^{d \times d})$ with respect to the basis $\{E_{ij}\}$.  Then, the condition number of $M'$ satisfies

%
$$\kappa \left( M' \right) ~<~ \max
\left\{ 144 \ee^2,~\frac{9 \ee^2}{4} \cdot (\delta -
1)^2 \right \},$$
and the smallest singular value of $M'$ satisfies
$$\sigma_{\rm min} \left( M' \right) > \frac{7}{20 a} \cdot \ee^{-(\delta + 1)/a} > \frac{C}{\delta}$$
for an absolute constant $C \in \R^+$.  Furthermore, $M'$ can be
inverted in $\mathcal{O} \left( \delta \cdot d \log d \right)$-time.
\label{thm:WellCondMeas}
\end{thm}

This theorem indicates that one can both efficiently and stably solve for $\x_0 \x_0^*$ using \eqref{eq:lifted_system} with the measurements given in \eqref{eq:MeasDef}.   This measurement scheme is also interesting because it corresponds to a ptychography system if we take the illumination function in \eqref{eq:1d_ptycho} to be $\widetilde{m}_n = \frac{\ee^{- i / a}}{\sqrt[4]{2 \delta - 1}}$ and assume that $d = k(2 \delta - 1)$ for some $k \in \N$. Then we may take the subset of the measurements \eqref{eq:ptychm} given by $j = (p -1)k + 1, \ p \in [2 \delta - 1]$ to obtain the masks specified in \eqref{eq:MeasDef}.  Consider that these assumptions may easily be met by zero-padding $\x_0$ until $d = k(2 \delta - 1)$ and simply ``throwing away'' all the measurements but those that correspond to $j = (p -1)k + 1, \ p\in [2\delta-1]$.

\subsection*{Example 2:} We provide a second deterministic construction that improves on the condition number of the previous collection of measurement vectors.  We merely set $\m_1 = e_1, \m_{2j} = e_1 + e_{j+1}$, and 
$\m_{2j + 1} = e_1 + i e_{j+1}$ for $j = 1, \ldots, \delta - 1$.  A simple induction shows that $\{S_\ell \m_j \m_j^* S_\ell^*\}_{\ell \in [d]_0, j \in [2k - 1]}$ is a basis for $T_k(\C^{d \times d})$, so if we take $\m_1, \ldots, \m_{2\delta - 1}$ for our masks we'll have a basis for $T_{\delta}(\C^{d \times d})$.  Indeed, if we let $$\mathcal{B} : T_k(\C^{d \times d}) \to \C^{\delta \times d}$$ be the measurement operator defined via 
$$\big(\mathcal{B}(X)\big)_{\ell, j} = \langle S_\ell \m_j \m_j^* S_\ell^*, X \rangle,$$
we can immediately solve for the entries of $X \in T_k(\H^{d \times d})$ from $\mathcal{B}(X) =: B$ by observing that \[\begin{array}{rcl}
X_{i,i} & = & B_{i - 1,1} \\
X_{i, i + k} & = & \frac{1}{2}B_{i - 1, 2k} + \frac{i}{2}B_{i - 1, 2k + 1} - \frac{1 + i}{2} (B_{i - 1, 1} + B_{i + k - 1, 1}), \end{array}\] where we naturally take the indices of $B$ mod $d$.  This leads to an upper triangular system if we enumerate $X$ by its diagonals; namely we regard $T_\delta(\H^{d \times d})$ as a $d(2 \delta - 1)$ dimensional vector space over $\R$ and set, for $i \in  [d]$ 
\[\begin{array}{ll}
z_{kd + i} =  \left\{\begin{array}{r@{,\quad}l} \operatorname{Re}(X_{i, i + k}) & 0 \le k < \delta \\ \operatorname{Im}(X_{i, i + k - \delta + 1}) & \delta \le k < 2\delta - 1 \end{array}\right., \ \ 
y_{kd + i}  =  \left\{\begin{array}{r@{,\quad}l} \mathcal{B}(X)_{i, 1} & k = 0 \\ \mathcal{B}(X)_{i, 2k} & 1 \le k < \delta \\  \mathcal{B}(X)_{i, 2(k - \delta + 1) + 1} & \delta \le k < 2 \delta - 1 \end{array}\right.
\end{array}. \] Then %
%
\[y = \begin{bmatrix} I_d & 0 & 0 \\ D & 2 I_{d(\delta - 1)} & 0 \\ D & 0 & 2 I_{d(\delta - 1)} \end{bmatrix}z =: Cz, \ \text{where} \ D = \begin{bmatrix} I_d + S \\ I_d + S^2 \\ \vdots \\ I_d + S^{\delta - 1} \end{bmatrix}.\]  Here the matrix $S$ is a $d\times d$ matrix representing the circular shift by one, $S_1$.  Since the matrix $C$ is upper triangular, its inverse is immediate: \[C^{-1} = \begin{bmatrix} I_d & 0 & 0 \\ -D / 2 & I_{d(\delta-1)} / 2 & 0 \\ -D / 2 & 0 & I_{d(\delta-1)} / 2 \end{bmatrix}.\]  To ascertain the condition number of $\mathcal{B}$
, then, all we need is the extremal singular values of $C$.  We bound the top singular value by considering 
\[\begin{array}{rcl} \sigma_{\max}(C) = \max\limits_{||w||^2 + ||v||^2 = 1} \left\lVert C\begin{bmatrix} w \\ v \end{bmatrix} \right\rVert  & = & \left\lVert \begin{bmatrix} w \\ D w \\ D w \end{bmatrix} + 2\begin{bmatrix} 0 \\ v \end{bmatrix} \right\rVert \\
& \le & \sqrt{||w||^2 + 2 ||w + Sw||^2 + \cdots + 2 ||w + S^{\delta - 1} w||^2} + ||2 v|| \\
& \le & 2\sqrt{2 \delta}||w|| + 2 ||v|| \le 2 + 2\sqrt{2\delta}.  \end{array}\]  By a nearly identical argument, we find \[\dfrac{1}{\sigma_{\min}(C)} = \sigma_{\max}(C^{-1}) \le 1/2 + \sqrt{2 \delta}\] so that the condition number is bounded by $\kappa(C) \le c\delta$ for some absolute constant $c$.

%% file: Spectrum_V2.tex

Consider line 3 of Algorithm~\ref{alg:phaseRetrieval1}, which shows that we are trying to recover $\widetilde{\x}_0:= \frac{{\x_0}}{|\x_0|}$ via an eigenvector method. Here, we show that  $\widetilde{X}_0$ has $\widetilde{\x}_0$ as its top eigenvector and we investigate the spectral properties of $\widetilde{X}_0$ in this section.

To begin, consider $U = T_\delta(\mathbbm{1}\mathbbm{1}^*)$, i.e., 
\begin{equation}
U_{j,k} =  \left\{ \begin{array}{ll} 1 & \textrm{if}~| j - k |~{\rm mod}~d  < \delta \\ 0 & \textrm{otherwise} \end{array} \right..
\label{equ:NormedUmatSPG}
\end{equation}
Observe that $U$ is circulant for all $\delta$, so its eigenvectors are always discrete Fourier vectors.  Setting $\omega_j = \ee^{2\pi\ii\frac{j-1}{d}}$ for $j = 1, 2, \ldots, d$, one can also see that the eigenvalues of $U$ are given by 
\begin{equation}
\nu_j = \sum_{k = 1}^d (U)_{1,k}\omega_j^{k-1} = 1 + \sum_{k = 1}^{\delta - 1}\omega_j^k + \omega_j^{-k} ~=~ 1 + 2 \sum_{k = 1}^{\delta - 1}\cos\left(\bigfrac{2\pi (j-1) k}{d}\right),
\label{eqn:lambda}
\end{equation}
for all $j = 1, \dots, d$.  In particular, $\nu_1 = 2 \delta - 1$.  Set $\Lambda = \diag\{\nu_1, \ldots, \nu_{d}\}$ and let $F$ denote the unitary $d \times d$ discrete Fourier matrix with entries 
$$F_{j,k} := \frac{1}{\sqrt{d}} \ee^{2\pi\ii\frac{(j-1)(k-1)}{d}},$$ then $U = F \Lambda F^*$.

We consider that $\widetilde{X}_0$ and $U$ are similar; indeed $\widetilde{\X}_0 = \widetilde{D}_0 U \widetilde{D}_0^*$, where $\widetilde{D}_0 = \diag\{(\widetilde{x}_0)_1, \ldots, (\widetilde{x}_0)_d\}$.  Since $|(\widetilde{\x}_0)_j| = 1$ for each $j$, we have that $\widetilde{D}_0$ is unitary.  Thus the eigenvalues of $\widetilde{X}_0$ are given by \eqref{eqn:lambda}, and its eigenvectors are simply the discrete Fourier vectors modulated by the entries of $\widetilde{\x}_0$.  %
We now have the following lemma.

\begin{lem}
Let $\widetilde{\X}_0$ be defined as in \eqref{eq:X_0}.  Then $$\widetilde{\X}_0 = \widetilde{D}_0 F \Lambda F^* \widetilde{D}_0^*$$ where $F$ is the unitary $d \times d$ discrete Fourier transform matrix, $\widetilde{D}_0$ is the $d \times d$ diagonal matrix $\diag\{(\widetilde{x}_0)_1, \ldots, (\widetilde{x}_0)_d\}$, and $\Lambda$ is the $d \times d$ diagonal matrix $\diag\{\nu_1, \ldots, \nu_{d}\}$ where
$$\nu_j := 1 + 2 \sum_{k = 1}^{\delta - 1}\cos\left(\bigfrac{2\pi (j-1) k}{d}\right)$$
for $j = 1, \dots, d$.
\label{lem:spectrum}
\end{lem}

We next estimate the principal eigenvalue gap of $\widetilde{\X}_0$.  This information will be crucial to our understanding of the stability and robustness of Algorithm~\ref{alg:phaseRetrieval1}.  

\subsection{The Spectral Gap of $\widetilde{\X}_0$}

Set $\theta_j = \frac{2 \pi j}{d}$ and begin by observing that, for any $\theta \in \R$,
\begin{eqnarray*}
  \sum_{k = 1}^{\delta - 1}\cos(\theta k) 
  & = & \bigfrac{1}{2}\left(\bigfrac{\sin(\theta(\delta - 1/2))}{\sin(\theta / 2)} - 1\right).
\end{eqnarray*}
Accordingly, defining $l_\delta : \R \to \R$ by $l_\delta(\theta) := 1 + 2\sum_{k=1}^{\delta - 1} \cos(\theta k)$ we have that
\begin{eqnarray}
  \nu_{j+1} 
  & = & l_\delta(\theta_j) ~=~ \bigfrac{\sin(\theta_j(\delta - 1/2))}{\sin(\theta_j / 2)}. \label{equ:EigenFormula}
\end{eqnarray}
Thus, the eigenvalues of $\widetilde{\X}_0$ are sampled from the $(\delta-1)^{\rm st}$ Dirichlet kernel.  
Of course, $\nu_1 = 2\delta - 1$ is the largest of these in magnitude, so the eigenvalue gap $\min_j \nu_1 - |\nu_j|$ is at most equal to
\begin{eqnarray*}
  \nu_1 - \nu_2 & = & (2\delta -  1) - \bigfrac{\sin(\pi / d (2\delta - 1))}{\sin(\pi / d)} \\
  & \le & (2 \delta - 1) - \bigfrac{\pi / d(2 \delta - 1) - \frac{1}{6}(\pi / d (2 \delta - 1))^3}{\pi / d} \\
  & = & \frac{1}{6}\left(\frac{\pi}{d}\right)^2(2 \delta - 1)^3 \label{equ:gap12}.
\end{eqnarray*}
Thus, $\nu_1 - |\nu_2 | \lesssim \frac{\delta^3}{d^2}$.  However, a lower bound on the spectral gap is more useful.  The following lemma establishes that the spectral gap is indeed $\sim \frac{\delta^3}{d^2}$ for most reasonable choices of $\delta < d$.

\begin{lem}
Let $\nu_1 = 2 \delta -1, \nu_2, \dots, \nu_d$ be the eigenvalues of $\widetilde{\X}_0$.  Then, there exists a positive absolute constant $C \in \R^+$ such that
$$\min_{j \in \{2, 3, \dots, d \} } (\nu_1 - |\nu_j| ) \geq C \frac{\delta^3}{d^2}$$
whenever $d \geq 4\delta$ and $\delta \geq 3$. 
\label{lem:EigGap}
\end{lem}

\begin{proof}

Let $\theta_j = \frac{2 \pi j}{d}$.  We find the lower bound by considering that $\theta_j \in [\pi / d, 2\pi - \pi / d]$ for every $j > 0$, so 
\[\nu_1 - \max | \nu_j | \ge \nu_1 - \max_{\theta \in [\pi / d, 2\pi - \pi / d]} | l_\delta(\theta) | = (2 \delta - 1) - \max_{\theta \in [\pi / d, \pi]} | l_\delta(\theta) |,\]where we have used our eigenvalue formula from \eqref{equ:EigenFormula}, and the symmetry of $l_\delta$ about $\theta = \pi$.

We now show that $l_\delta$ is decreasing towards its first zero at $\theta = \frac{2 \pi}{2 \delta - 1}$ by considering the derivative \[l_\delta'(\theta) = \bigfrac{(\delta - 1/2)\cos((\delta - 1/2)\theta) \sin(\theta/2) - 1/2 \sin((\delta - 1/2)\theta) \cos(\theta / 2)}{\sin(\theta / 2) ^ 2},\] which is non-positive if and only if \[(2\delta - 1) \sin(\theta/2)\cos((\delta - 1/2)\theta)  \le  \sin((\delta - 1/2)\theta) \cos(\theta /2 ).\]  Since $\tan(\cdot)$ is convex on $[0,\, \pi / 2)$, this last inequality will hold for $\theta \in [0,\, \frac{\pi}{2\delta - 1})$.  For $\theta \in [\frac{\pi}{2\delta-1}, \frac{2\pi}{2\delta -1})$, $\cos((\delta - 1/2)\theta) \leq 0$ while the remainder of the terms are non-negative, so the inequality also holds.  Therefore,
\[\nu_1 - \max_{j > 1} |\nu_j| \ge (2 \delta - 1) - \max \left\{\nu_2, \max_{\theta \in [\frac{2\pi}{2\delta-1}, \pi]} |l_\delta(\theta)| \right\},\]
which permits us to bound $(2 \delta - 1) - \nu_2$ and $(2 \delta - 1) - \max\limits_{\theta \in [\frac{2\pi}{2\delta-1}, \pi]} |l_\delta(\theta)|$ separately.


For $\max\limits_{\theta \in [\frac{2\pi}{2\delta-1}, \pi]} |l_\delta(\theta)|$, we simply observe that \[\max\limits_{\theta \in [\frac{2\pi}{2\delta-1}, \pi]} |l_\delta(\theta)| \le \max\limits_{\theta \in [\frac{2\pi}{2\delta-1}, \pi]} \dfrac{1}{\sin(\theta / 2)} = \left( \sin \left(\frac{2\pi}{2\delta - 1} \right) \right)^{-1}  \le \bigfrac{2\delta - 1}{4},\] where the last line uses that $\bigfrac{2 \pi}{2\delta - 1} \le \pi / 2$ (since $\delta \ge 3$).  This yields $\nu_1 - \max\limits_{\theta \in [\frac{2\pi}{2\delta-1}, \pi]} |l_\delta(\theta)| \ge \frac{3}{4}(2 \delta - 1)$.

As for $\nu_2$, we have $\theta_1 \cdot (\delta - 1) \le \pi / 2$ (since $4(\delta - 1) \le d$). Thus, $\cos(\cdot)$ will be concave on $[0,\, \theta_1(\delta-1)]$.  Considering \eqref{eqn:lambda}, this will give $\sum_{k=1}^{\delta - 1}\cos(k\theta_1) \le (\delta - 1) \cos \left(\theta_1 \frac{\delta}{2} \right)$, so \[\begin{array}{rcl}
  \nu_1 - \nu_2 & \ge & 2(\delta - 1) \left(1 - \cos \left(\pi\frac{\delta}{d} \right) \right)\\
  & \ge & 2(\delta - 1) \left(\dfrac{(\pi \frac{\delta}{d})^2}{4}\right) \\
  & \ge & \dfrac{\pi^2}{3} \cdot \dfrac{\delta^3}{d^2}.
\end{array}\]
The stated result follows.  
\end{proof}

We are now sufficiently well informed about $\widetilde{\X}_0$ to consider perturbation results for its leading eigenvector.

%% file: AltPerturb_V3.tex
 In this section we will use spectral graph theoretic techniques to obtain a bound on the error associated with recovering phase information using our method.  
In particular, we will adapt the proof of Theorem 6.3 from \cite{alexeev2014phase} to develop a bound for $\min_{\theta\in[0,2\pi]} \Vert \tx_0 - \mathbbm{e}^{\mathbbm{i}\theta}\tx \Vert_2$.  This approach involves considering both $\tX$ from Algorithm~\ref{alg:phaseRetrieval1} and $\widetilde{\X}_0$ from \eqref{eq:X_0} in the context of spectral graph theory, so we begin by defining essential terms.  The idea is to consider a graph whose vertices correspond to the entries of $\tx_0$ from \eqref{eq:tx_0}, and whose edges carry the relative phase data.\footnote{The interested reader is also referred to the appendix where more standard perturbation theoretic techniques are utilized in order to obtain a  weaker bound on the error associated with recovering phase information via the proposed approach.}
\\

We begin with an undirected graph $G = (V, E)$ with vertex set $V = \{1, 2, \dots, d\}$ and weight mapping $w: V \times V \to \R^+$, where $w_{ij} = w_{ji}$ and $w_{ij} = 0$ iff $\{i, j\} \notin E$.  The \textbf{degree} of a vertex $i$ is \[\deg(i) := \sum_{j ~{\rm s.t.}~ (i, j) \in E} w_{ij},\] and we define the \textbf{degree matrix} and \textbf{weighted adjacency matrix} of $G$ by \[D := \diag(\deg(i)) \ \text{and} \ W_{ij} := w_{ij},\] respectively.  The \textbf{volume} of $G$ is \[\vol(G) := \sum_{i \in V} \deg(i).\]  Finally, the \textbf{Laplacian} of $G$ is the $d \times d$ real symmetric matrix \[L := I - D^{-1/2} W D^{-1/2} = D^{-1/2}(D - W)D^{-1/2},\]
where $I \in \{ 0,1\}^{d \times d}$ is the identity matrix.  

When $G$ is connected, Lemma 1.7 of \cite{chungspectral} shows that the nullspace of $(D - W)$ is $\Span(\mathbbm{1})$, and the nullspace of $L$ is $\Span(D^{1/2}\mathbbm{1})$.  Observing that $D - W$ is diagonally semi-dominant, it follows from Gershgorin's disc theorem that $(D - W)$ and $L$ are both positive semidefinite.  Alternatively, one may also note that 
\[{\bf v}^*(D - W){\bf v} = \sum_{i \in V} \left(v_i^2\deg(i) - \sum_{j \in V} v_i v_j w_{ij}\right) = \frac{1}{2} \sum_{i, j \in V} w_{ij} (v_i - v_j)^2 \geq 0\]
holds for all ${\bf v} \in \mathbbm{R}^d$.  Thus, we may order the eigenvalues of $L$ in increasing order so that $0 = \lambda'_1 < \lambda'_2 \le \cdots \le \lambda'_n$.  
We then define the \textbf{spectral gap} of $G$ to be $\tau = \lambda'_2$.

Herein, though we will state the main theorem of this section more generally, we will only be interested in the case where the graph $G=(V,E)$ is the simple unweighted graph whose adjacency matrix is $\widetilde{U}$ from \eqref{equ:NormedUmatSPG}.  In this case we will have $W = \widetilde{U}$ and $D = (2 \delta - 1)I$.  We also immediately obtain the following corollary of Lemmas~\ref{lem:spectrum} and~\ref{lem:EigGap}.

\begin{cor}
Let $G$ be the simple unweighted graph whose adjacency matrix is $\widetilde{U}$ from \eqref{equ:NormedUmatSPG}.  Let $L$ be the Laplacian of $G$.  Then, there exists a bijection $\sigma:[d] \rightarrow [d]$ such that
$$\lambda'_{\sigma(j)} = 1 - \frac{1 + 2 \sum_{k = 1}^{\delta - 1}\cos\left(\bigfrac{2\pi (j-1) k}{d}\right)}{2 \delta - 1}$$
for $j = 1, \dots, d$.  In particular, if $d \geq 4(\delta - 1)$ and $\delta \geq 3$ then $\tau = \lambda'_2 > C'''\delta^2 / d^2$ for an absolute constant $C''' \in \mathbbm{R}^+$.
\label{cor:Gspectrum}
\end{cor}

Using this graph $G$ as a scaffold we can now represent our computed relative phase matrix $\tX$ from Algorithm~\ref{alg:phaseRetrieval1} by noting that for some (Hermitian) perturbations $\eta_{ij}$ 
we will have 
\begin{equation}
\tX_{ij} = \bigfrac{(x_0)_i(x_0)_j^* + \eta_{ij}}{|(x_0)_i(x_0)_j^* + \eta_{ij}|} \cdot w_{ij} = \bigfrac{(x_0)_i(x_0)_j^* + \eta_{ij}}{|(x_0)_i(x_0)_j^* + \eta_{ij}|} \cdot \chi_{E(i, j)}.
\label{equ:DeftildeXviaGcomp}
\end{equation}
 Using this same notation we may also represent our original phase matrix $\widetilde{\X}_0$ 
via $G$ by noting that
\begin{equation}
(\widetilde{\X}_0)_{ij} = \bigfrac{(x_0)_i(x_0)_j^*}{|(x_0)_i(x_0)_j^*|} \cdot w_{ij} = \sgn \left( (x_0)_i(x_0)_j^* \right) \cdot \chi_{E(i, j)}.
\label{equ:DeftildeXviaGpure}
\end{equation}


We may now define the \textbf{connection Laplacian} of the graph $G$ associated with the Hermitian and entrywise normalized data given by $\tX$ to be the matrix 
\begin{equation}
L_1 = I - D^{-1/2} (\tX \circ W) D^{-1/2},
\label{equ:ConnectLaplace}
\end{equation}
where $\circ$ denotes entrywise (Hadamard) multiplication.
Following \cite{Cheeger}, given $\tX$ and a vector ${\bf y} \in \mathbbm{C}^d$, we  define the \textbf{frustration of} $\bf y$ \textbf{with respect to} $\tX$ by \begin{equation} \eta_{\tX}({\bf y}) := \bigfrac{1}{2}\bigfrac{\sum_{(i, j) \in E} w_{ij}|y_i - \tX_{ij} y_j |^2}{\sum_{i \in V} \deg(i) |y_i|^2} = \bigfrac{{\bf y}^* (D - (\tX \circ W)) {\bf y}}{{\bf y}^* D {\bf y}}. \label{eq:frustration} \end{equation}
We may consider $\eta_{\tX}({\bf y})$ to measure how well ${\bf y}$ (viewed as a map from $V$ to $\mathbbm{C}$) conforms to the computed relative phase differences $\tX$ across the graph $G$. 

%



In addition, we adapt a result from \cite{Cheeger}:

\begin{lem}[Cheeger inequality for the connection Laplacian]  
Suppose that $G = (V=[d], E)$ is a connected graph with degree matrix $D \in [0,\infty)^{d \times d}$, weighted adjacency matrix $W \in [0,\infty)^{d \times d}$, and spectral gap $\tau > 0$, and that $\tX \in \mathbbm{C}^{d \times d}$ is Hermitian and entrywise normalized.  Let ${\bf u} \in \mathbbm{C}^d$ be an eigenvector of $L_1$ from \eqref{equ:ConnectLaplace} corresponding to its smallest eigenvalue.  Then, ${\bf w} = \sgn({\bf u}) = \sgn \left( D^{-1/2}{\bf u} \right)$ satisfies \[\eta_{\tX}({\bf w}) \le \bigfrac{C'}{\tau} \cdot \min_{{\bf y} \in \mathbbm{C}^d} \eta_{\tX}( \sgn({\bf y}) ),\] where $C' \in \mathbbm{R}^+$ is a universal constant.
\label{lem:CheegerInequality}
\end{lem}

\begin{proof}  
 One can see that 
\begin{align*}
\inf_{{\bf v} \in \C^d \setminus \{ \bf 0 \}} \bigfrac{{\bf v}^* L_1 {\bf v}}{{\bf v}^* {\bf v}} &= \inf_{{\bf y} \in \C^d \setminus \{ \bf 0 \}} \bigfrac{(D^{1/2}{\bf y})^* L_1 (D^{1/2} {\bf y})}{(D^{1/2}{\bf y})^*(D^{1/2} {\bf y})} = \inf_{{\bf y} \in \C^d \setminus \{ \bf 0 \} } \bigfrac{{\bf y}^*(D - (\tX \circ W)){\bf y}}{{\bf y}^* D {\bf y} } \\
&= \inf_{y \in \C^d \setminus \{ \bf 0 \}} \eta_{\tX}({\bf y}) \le \min_{{\bf y} \in \mathbbm{C}^d} \eta_{\tX}( \sgn({\bf y}) ).  
\end{align*}
From here, Lemma 3.6 in \cite{Cheeger} gives \[\eta_{\tX}({\bf w}) \le \bigfrac{44}{\tau} \eta_{\tX}\left( D^{-1/2}{\bf u} \right) = \bigfrac{44}{\tau} \cdot \inf_{{\bf v} \in \C^d \setminus \{ \bf 0 \}} \bigfrac{{\bf v}^* L_1 {\bf v}}{{\bf v}^* {\bf v}} \le \bigfrac{44}{\tau} \cdot \min_{{\bf y} \in \mathbbm{C}^d} \eta_{\tX}( \sgn({\bf y}) ).\]
\end{proof}

We now state the main result of this section:

\begin{thm}
Suppose that $G = (V=[d], E)$ is an undirected, connected, and unweighted graph (so that $W_{ij} = \chi_{E(i, j)}$) with spectral gap $\tau > 0$.  Let ${\bf u} \in \mathbbm{C}^d$ be an eigenvector of $L_1$ from \eqref{equ:ConnectLaplace} corresponding to its smallest eigenvalue, and let \[\widetilde{\bf x} = \sgn({\bf u}) \ \text{and} \ \tx_0 = \sgn(\x_0).\] Then for some universal constant $C \in \mathbbm{R}^+$, \[\min_{\theta \in [0, 2\pi]} ||\tx - \ee^{\ii \theta} \tx_0||_2 \le C \bigfrac{\|\tX - \widetilde{\X}_0\|_F}{\tau \cdot \sqrt{\min_{i \in V}(\deg(i))}},\] where $\tX$ and $\widetilde{\X}_0$ are defined as per \eqref{equ:DeftildeXviaGcomp} and \eqref{equ:DeftildeXviaGpure}, respectively. 
\label{thm:SpecGraphPertBound}
\end{thm}

The proof follows by combining the two following lemmas, which share the hypotheses of the theorem.  Additionally, we introduce the notation ${\bf g} \in \C^d$ and $\Lambda \in \C^{d \times d}$, where $$g_i = (\tx_0)^*_i \tx_i \quad \text{and} \quad \Lambda_{ij} = (\tX_0)^*_{ij} \tX_{ij},$$ and observe that $|g_i| = |\Lambda_{ij}| = 1$ for each $(i, j) \in E$.

\begin{lem}
Under the hypotheses of Theorem \ref{thm:SpecGraphPertBound}, there exists an angle $\theta \in [0, 2\pi]$ such that \[\tau \sum_{i \in V} \deg(i) \lvert g_i - \ee^{\ii \theta} \rvert^2 \le 2 \sum_{(i, j) \in E} \lvert g_i - g_j \rvert^2.\]
\label{lem:gbound1}
\end{lem}

\begin{lem}
Under the hypotheses of Theorem \ref{thm:SpecGraphPertBound}, there exists an absolute constant $C$ such that \[2\sum_{(i, j) \in E} \lvert g_i - g_j \rvert^2 \le \dfrac{C}{\tau} \lVert \tX - \tX_0 \rVert^2_F.\]
\label{lem:gbound2}
\end{lem}

From these lemmas, the theorem follows immediately by observing $\sum_{i \in V} |g_i - \ee^{\ii \theta}|^2 = \lvert \tx - \ee^{\ii \theta} \tx_0 \rvert_2^2$.

\begin{proof}[Proof of Lemma~\ref{lem:gbound1}]
We set $\alpha = \bigfrac{\sum_{i \in V} \deg(i) g_i}{\vol(G)}$ and $w_i = g_i - \alpha$.  Then \[\mathbbm{1}^* D {\bf w} = \sum_{i \in V} \deg(i)(g_i - \alpha) = 0,\] so $D^{1/2}{\bf w}$ is orthogonal to $ D^{1/2}\mathbbm{1}$.  Noting that the null space of $L$ is spanned by $D^{1/2}\mathbbm{1}$ when $\tau > 0$, and recalling that $L \succeq 0$, we have \[\bigfrac{(D^{1/2}{\bf w})^* L (D^{1/2} {\bf w})}{{\bf w}^* D {\bf w}} \ge \min_{ \y^*D^{1/2}\mathbbm{1}=0} \bigfrac{{\bf y}^* L {\bf y}}{{\bf y}^* {\bf y}} = \tau.\]  Therefore, \[\everymath{\displaystyle} 
\begin{array}{rclcl}%
\tau {\bf w}^* D {\bf w} & \le&  {\bf w}^*(D - W) {\bf w} & = & {\bf g}^*(D - W) {\bf g}  \\
 & = & \sum_{i \in V} \deg(i) |g_i|^2 - \sum_{i \in V} g_i^* \sum_{(i,j) \in E}  g_j & = & \sum_{(i,j) \in E} (1 - g _i^* g_j) \\
 & = & \frac{1}{2} \sum_{(i, j) \in E} |g_i - g_j|^2.%
\end{array}\]%
We note that $\tau {\bf w}^* D {\bf w} = \tau \sum_{i \in V}\deg(i)|g_i - \alpha|^2$, while we seek a bound on $\sum_{i \in V} \deg(i) |g_i - \ee^{\ii \theta}|^2$.  To that end, we use the fact that $|g_i| = |\sgn(\alpha)| = 1$ to obtain $$|g_i - \sgn(\alpha)| \le |g_i - \alpha| + |\alpha - \sgn(\alpha)| \le 2|g_i - \alpha|.$$   Setting $\theta := \arg{\alpha}$, we have the stated result.
\end{proof}

\begin{proof}[Proof of Lemma~\ref{lem:gbound2}]
Observe that for any two real numbers $a, b \in \R$, we have $\frac{1}{2} a^2 - b^2 \le (a - b)^2$. Thus, by the reverse triangle inequality we have 
\[\everymath{\displaystyle}\begin{array}{rcl}%
\sum_{(i, j) \in E}  \left(\frac{1}{2} |g_i - g_j|^2 - |\Lambda_{ij} - 1|^2\right) & \le & \sum_{(i, j) \in E}  \left( |g_i - g_j| - |\Lambda_{ij} - 1|\right)^2 \\%
 & \le & \sum_{(i, j) \in E}  |g_i - \Lambda_{ij}g_j|^2 \\%
& = & \sum_{(i, j) \in E}  |\tx_i - \tX_{ij} \tx_j |^2 \\%
& = & 2 \vol(G) \cdot \eta_{\tX}(\tx),%
\end{array}\]%
as the denominator of \eqref{eq:frustration} is $2\vol(G)$ whenever the entries of ${\bf y}$ all have unit modulus.

Lemma~\ref{lem:CheegerInequality} now tells us that 
\[ \sum_{(i, j) \in E}  \left(\frac{1}{2} |g_i - g_j|^2 - |\Lambda_{ij} - 1|^2\right) \le \bigfrac{2 C' \vol(G)}{\tau} \min_{{\bf y} \in \mathbbm{C}^d} \eta_{\tX}( \sgn({\bf y}) ) \le \bigfrac{2C' \vol(G)}{\tau} \eta_{\tX}(\tx_0).\] 
Moreover,%
$$\begin{array}{rcl}%
\eta_{\tX}(\tx_0) & = & \bigfrac{\sum_{(i, j) \in E} |(\tx_0)_i - \tX_{ij} (\tx_0)_j |^2}{2 \sum_{i \in V} \deg(i) |(\tx_0)_i|^2} \\%
& = & \bigfrac{\sum_{(i, j) \in E} |(\tx_0)_i (\tx_0)^*_j - \tX_{ij}|^2}{2 \vol(G)} \\%
& = & \bigfrac{ \| \widetilde{\X}_0 - \tX \|_F^2}{2 \vol(G)},%
\end{array}$$%
so that $\sum_{(i, j) \in E} \frac{1}{2} |g_i - g_j|^2 \le \frac{C'}{\tau} \lVert \X_0 - \X \rVert^2_F + \sum_{(i, j) \in E} |\Lambda_{ij} - 1|^2$.  Considering also that $$\sum_{(i, j) \in E} |\Lambda_{ij} - 1|^2 = \sum_{(i, j) \in E} |\tX_{ij} - (\tX_0)_{ij}|^2 = \| \tX - \tX_0 \|_F^2$$ and $\tau \le 1$, this completes the proof.
\end{proof}

We may now use Theorem~\ref{thm:SpecGraphPertBound} to produce a perturbation bound for our banded matrix of phase differences $\widetilde{\X}_0$.

\begin{cor}
Let $\tX_0$ be the matrix in \eqref{eq:X_0}, $\tx_0$ be the vector of true phases \eqref{eq:tx_0}, and $\tX$ be as in line 3 of Algorithm~\ref{alg:phaseRetrieval1} with $\tx = \sgn({\bf u})$ where ${\bf u}$ is the top eigenvector of $\widetilde{X}$. Suppose that 
$\Vert \tX_0 - \tX \Vert_F \le \eta \Vert \tX_0 \Vert_F$ for some $\eta>0$.  Then, there exists an absolute constant $C' \in \mathbb{R}^+$ such that
\[\min_{\theta\in[0,2\pi]} \Vert \tx_0 - \mathbbm{e}^{\mathbbm{i}\theta} \tx \Vert_2 \le C' \bigfrac{\eta d^\frac{5}{2}}{\delta^{2}}.\]
\label{cor:GenBoundv2}
\end{cor}

\begin{proof}
 We apply Theorem~\ref{thm:SpecGraphPertBound} with the unweighted and undirected graph $G = (V, E)$, where $V = [d]$ and $E = \{(i, j) : |i - j| \mod d < \delta\}$.  Observe that $G$ is also connected and $(2\delta - 1)$-regular so that $\min_{i \in V}(\deg(i)) = 2\delta - 1$.  The spectral gap of $G$ is $\tau > C'''\delta^2 / d^2 > 0$ by Corollary~\ref{cor:Gspectrum}.  We know that $\| \widetilde{X}_0 \|_F = \sqrt{d(2\delta - 1)}$, so that $\| \tX_0 - \tX \|_{F} \le C'' \eta (d \delta)^{1/2}$.  Finally, if ${\bf u}$ is the top eigenvector of $\widetilde{X}$ then it will also be an eigenvector of $L_1$ corresponding to its smallest eigenvalue since, here, $L_1 = I - \frac{1}{2 \delta - 1} \tX $.

Combining these observations we have \[\min_{\theta\in[0,2\pi]} \Vert \tx_0 - \mathbbm{e}^{\mathbbm{i}\theta} \tx \Vert_2 \le C \bigfrac{C'' \eta (d \delta)^{1/2}}{C''' \delta^2 / d^2 \cdot (2 \delta - 1)^{1/2}} = C' \bigfrac{\eta d^{5/2}}{\delta^2}.\]
\end{proof}

%% file: RecovGuarantee_V2.tex
Herein we will assume Algorithm~\ref{alg:phaseRetrieval1} is provided with measurements $\y$ of the form \eqref{eq:shift_model_noise} such that the linear operator \eqref{eq:linear} is invertible on $T_{\delta}(\C^{d \times d})$ with condition number $\kappa > 0$.  Unless otherwise stated, we follow the notation of \S \ref{sec:locCorrMeas}- \S\ref{sec:mainRes}; therefore, our assumptions imply that $\lVert \X - \X_0 \rVert_F \le \kappa \lVert \n \rVert_2$.

We now aim to bound the Frobenius norm of the perturbation error $(\tilde{X} - \tilde{\X}_0)$ present in the matrix $\tilde{X}$ formed in line 2 of Algorithm~\ref{alg:phaseRetrieval1}.
Toward this end we define the set of $\rho$-small indexes of $\x_0$ to be 
\begin{equation}
S_\rho := \left\{ j~\bigg|~ | \left( x_0 \right)_j | < \left(\frac{\delta \| \n \|_2}{\rho} \right)^{\frac{1}{4}} \right\}
\label{equ:SrhoDef}
\end{equation}
where $\rho \in \mathbbm{R}^+$ is a free parameter.  
%
%
With the definition of $S_\rho$ in hand we can bound the perturbation error $(\tilde{X} - \tilde{\X}_0)$ using the next lemma.

\begin{lem}
Let $\tilde{\X}$ be the matrix computed in line 2 of Algorithm~\ref{alg:phaseRetrieval1}.  We have that $$\| \tilde{\X} - \tilde{\X}_0 \|_{\rm F} \leq C \sqrt{\frac{\rho \frac{\kappa}{\delta} \| \n \|_2 + \left| S_\rho \right|}{d}} \cdot \| \tilde{\X}_0 \|_{\rm F}$$
holds for all $\rho \in \R^+$, where $C$ is an absolute constant.
\label{lem:EtaBound}
\end{lem}

\begin{proof}
For any $j,k$ with 
$| j - k |~{\rm mod}~d  < \delta$ we have that 
$$|(\tilde{\X}_0)_{jk} - \tilde{\X}_{jk}| = |\mathbbm{e}^{\mathbbm{i}(\phi_{jk} - \beta_{jk})} - 1| = 2 \sin \left(  \frac{| \phi_{jk} - \beta_{jk} |}{2} \right)$$
where $\phi_{jk} = \arg (\tilde{\X}_0)_{jk}$ and $\beta_{jk} = \arg (\tilde{\X}_{j,k})$.  Defining $N_{jk} = \X_{jk} - (\X_0)_{jk}$, the law of sines now implies that $$2 \sin \left(  \frac{| \phi_{jk} - \beta_{jk} |}{2} \right) \leq 2 \left| \sin \left(  \frac{\phi_{jk} - \beta_{jk} }{2} \right) \right| \leq 2 \frac{|N_{jk}|}{| (\X_0)_{jk} |} \leq 2 \rho^{\frac{1}{2}} \frac{|N_{jk}|}{(\kappa \| \n \|_2)^{\frac{1}{2}}}$$ whenever $j,k \in S_\rho^{c}$.  Thus, there exists an absolute constant $C' \in \R^+$ such that%
\begin{align*}
\| \tilde{\X} - \tilde{\X}_0 \|^2_{\rm F} &\leq \sum_{j,k \in S_\rho^{c}} 4 \rho \frac{|N_{jk}|^2}{\kappa \| \n \|_2} + \sum_{j \in S_\rho, ~{\rm or}~ k \in S_\rho} |(\tilde{\X}_0)_{jk} - \tilde{\X}_{jk}|^2\\ &\leq 4 \rho \frac{\| N \|^2_{\rm F}}{\delta \| \n \|_2} + \sum_{j \in S_\rho} 4 \cdot (4\delta - 3) = 4 \rho \frac{\| N \|^2_{\rm F}}{\kappa \| \n \|_2} + 4 \cdot (4\delta - 3) \left| S_\rho \right| \\ & \leq C' (\rho \kappa \| \n \|_2 + \delta \left| S_\rho \right|).
\end{align*}%
The proof is completed by recalling that $\lVert \tX_0 \rVert_F = \sqrt{(2\delta - 1)d}$.
\end{proof}

We are finally ready to prove a robustness result for Algorithm~\ref{alg:phaseRetrieval1}.

\begin{thm}
Suppose that $\tilde{\X}$ and $\tilde{\X}_0$ satisfy $\Vert \tilde{\X} - \tilde{\X}_0 \Vert_{\rm F}~\leq \eta \Vert \tilde{\X}_0 \Vert_F$ for some $\eta>0$.
Then, the estimate $\x$ produced by Algorithm~\ref{alg:phaseRetrieval1} satisfies 
\[ \min_{\theta \in [0, 2 \pi]} \left\Vert  \x_0 - \mathbbm{e}^{\mathbbm{i} \theta} \x \right\Vert_2 \leq C \Vert \x_0 
        \Vert_{\infty} \left( \frac{d^{5/2}}{\delta^2} \right) \eta  + C  d^{\frac{1}{4}} \sqrt{\kappa \| \n \|_2 },\]
where $C \in \mathbb{R}^+$ is an absolute universal constant.  Alternatively, one can bound the error in terms of the size of the index set $S_\rho$ from \eqref{equ:SrhoDef} as 
\begin{equation}
\min_{\theta \in [0, 2 \pi]} \left\Vert  \x_0 - \mathbbm{e}^{\mathbbm{i} \theta} \x \right\Vert_2 \leq C' \Vert \x_0 
        \Vert_{\infty} \left( \frac{d}{\delta} \right)^2 \sqrt{\rho \frac{\kappa}{\delta} \| \n \|_2 + \left| S_\rho \right|} + C' d^{\frac{1}{4}} \sqrt{\kappa \| \n \|_2 }, \label{equ:ManRes2} 
\end{equation}
for any desired $\rho \in \R^+$, where $C' \in \mathbb{R}^+$ is another absolute universal constant.
\label{thm:MainRes}
\end{thm}

\begin{proof}

Let $\phi \in [0,2 \pi)$ be arbitrary;  
then $\mathbbm{e}^{\mathbbm{i} \phi }\x = |\x| \circ \mathbbm{e}^{\mathbbm{i} \phi} \tilde{\x}$ and $\x_0 = |\x_0| \circ \tilde{\x}_0$, where $\circ$ denotes the entrywise (Hadamard) product.

We see that%
\begin{align*}
\min_{\phi \in [0, 2 \pi]} \| \x_0 - \mathbbm{e}^{\mathbbm{i} \phi }\x \|_2 &= \min_{\phi \in [0, 2 \pi]} \left\| |\x_0| \circ \tilde{\x}_0 - |\x| \circ \mathbbm{e}^{\mathbbm{i} \phi} \tilde{\x} \right\|_2 \\
&\leq \min_{\phi \in [0, 2 \pi]} \left\| |\x_0| \circ \tilde{\x}_0 - |\x_0| \circ \mathbbm{e}^{\mathbbm{i} \phi} \tilde{\x}  \right\|_2 + \left\| |\x_0| \circ \mathbbm{e}^{\mathbbm{i} \phi} \tilde{\x}  - |\x| \circ \mathbbm{e}^{\mathbbm{i} \phi} \tilde{\x}  \right\|_2 
\end{align*}
where the second term is now independent of $\phi$.  As a result we have that
\begin{equation*}
\min_{\phi \in [0, 2 \pi]} \| \x_0 - \mathbbm{e}^{\mathbbm{i} \phi }\x \|_2 \leq \| \x_0 \|_{\infty} \left( \min_{\phi \in [0, 2 \pi]} \| \tilde{\x}_0 - \mathbbm{e}^{\mathbbm{i} \phi} \tilde{\x} \|_2 \right) + C'' \sqrt{\kappa \sqrt{d} \cdot \| \n \|_2 }
\end{equation*}
for some absolute constant $C'' \in \mathbbm{R}^+$.  Here the bound on the second term follows from Lemma 3 of \cite{IVW2015_FastPhase} and the Cauchy-Schwarz inequality.  
%
%
The first inequality of the theorem now results from an application of Corollary~\ref{cor:GenBoundv2} to the first term.  The second inequality then follows from Lemma~\ref{lem:EtaBound}.
\end{proof}

Looking at the second inequality \eqref{equ:ManRes2} in Theorem~\ref{thm:MainRes} we can see that the error bound there will be vacuous in most settings unless $S_\rho = \emptyset$.  Recalling \eqref{equ:SrhoDef}, one can see that $S_\rho$ will be empty as soon as $\rho = \kappa \delta \| \n \|_2 / \left| (x_0)_{\rm min} \right|^4$, where $(x_0)_{\rm min}$ is the smallest magnitude of any entry in $\x_0$.  Utilizing this value of $\rho$ in \eqref{equ:ManRes2} leads to the following corollary of Theorem~\ref{thm:MainRes}.

\begin{cor}
Let $(x_0)_{\rm min} := \min_j |(x_0)_j|$ be the smallest magnitude of any entry in $\x_0$.  Then, the estimate $\x$ produced by Algorithm~\ref{alg:phaseRetrieval1} satisfies 
\[ \min_{\theta \in [0, 2 \pi]} \left\Vert  \x_0 - \mathbbm{e}^{\mathbbm{i} \theta} \x \right\Vert_2 \leq C \left( \frac{\Vert \x_0 
        \Vert_{\infty}}{(x_0)^2_{\rm min}} \right) \left( \frac{d}{\delta} \right)^2 \kappa \| \n \|_2 + C d^{\frac{1}{4}} \sqrt{\kappa \| \n \|_2 },\]
where $C \in \mathbb{R}^+$ is an absolute universal constant.  
\label{Cor:RecovRes}
\end{cor}

Corollary~\ref{Cor:RecovRes} yields a deterministic recovery result for any signal $\x_0$ which contains no zero entries.  If desired, a randomized result can now be derived from Corollary~\ref{Cor:RecovRes} for arbitrary $\x_0$ by right multiplying the signal $\x_0$ 
with a random ``flattening'' matrix as done in \cite{IVW2015_FastPhase}.   Finally, we note that a trivial variant of Corollary~\ref{Cor:RecovRes} can also be combined with the discussion in \S \ref{sec:STFT} in order to generate recovery guarantees for the windowed Fourier measurements defined by \eqref{eq:STFT_measurements}.  However, we will leave such variants and extensions to the interested reader.

%% file: NumEval.tex
We now present numerical simulations supporting the theoretical recovery guarantees in Section
\ref{sec:RecovGuarantee}.  Our main objective is to evaluate the proposed algorithm against other
existing phase retrieval methods using {\em local} measurements. However, for completeness, we also
present selected results comparing the proposed formulation against other well established phase
retrieval algorithms (such as {\em Wirtinger Flow}) using {\em
global} measurements such as coded diffraction patterns (CDPs). The
results presented here may be recreated using the open source {\em BlockPR} Matlab software package
which is freely available at \cite{bitbucket_BlockPR}. Unless otherwise stated, we use i.i.d.
zero-mean complex Gaussian random test signals with measurement errors modeled using an aditive
Gaussian noise
model. Applied measurement noise and reconstruction error are both reported in decibels (dB) in terms of signal
to noise ratios (SNRs), with 
\[  \mbox{SNR (dB)} = 10 \log_{10} \left( 
        \frac{\sum_{j=1}^D\vert \langle \a_j, \x_0 \rangle \vert ^4}
            {D\sigma^2} \right), \qquad  
    \mbox{Error (dB)} = 10 \log_{10} \left( \frac{\|\x - 
	\x_0 \|_2^2}{\| \x_0 \|^2_2}\right),  \]
where $\a_j, \x_0, \x, \sigma^2$ and $D$ denote the measurement vectors, true signal, recovered signal,
(Gaussian) noise variance and number of measurements respectively. 
All simulations were performed on a laptop computer running GNU/Linux
(Ubuntu Linux 16.04 \verb|x86_64|) with an Intel$^\circledR$ Core\texttrademark M-5Y10c processor,
8GB RAM and Matlab R2016a. Each data point in the timing and robustness plots were obtained as the
average of $100$ trials.

\subsection{Numerical Improvements to Algorithm~\ref{alg:phaseRetrieval1}:  Magnitude Estimation}
\label{sec:MagEstImpNumerical}

Looking at the matrix $\X$ formed on line 1 of Algorithm~\ref{alg:phaseRetrieval1} one can see that
$$\X= \X_0 + N'$$
where $\X_0$ is the banded Hermitian matrix $T_\delta(\x_0 \x_0^*)$ defined in \eqref{equdef:X0}, and $N'$ contains arbitrary banded Hermitian noise.  As stated and analyzed above, Algorithm~\ref{alg:phaseRetrieval1} takes advantage of this structure in line 4 in order to estimate the magnitude of each entry of $\x_0$ based on the fact that 
$$\X_{jj} = |(x_0)_{j}|^2 + N'_{jj}$$
holds for all $j \in [d] := \{ 1, \dots, d\}$.  Though this magnitude estimate suffices for our theoretical treatment above, it can be improved on in practice by using slightly more general techniques.

Considering the component-wise magnitude of $X$, $|X| \in \mathbbm{R}^{d \times d}$, one can see that its entries are 
\begin{equation*}
|\X|_{jk} =  \left\{ \begin{array}{ll} |({x}_0)_j| |({x}_0)_{k}| + N''_{jk} & \textrm{if}~| j - k ~{\rm mod}~d | < \delta \\ 0 & \textrm{otherwise} \end{array} \right.,
\end{equation*}
where $N'' \in \mathbbm{R}^{d \times d}$ represents the changes in magnitude to the entries of $|\X_0|$ due to noise. 
We may then let $D_j \in \R^{\delta \times \delta}$ denote the submatrix of $|X|$ given by \[(D_j)_{kh} = |X|_{(j+k-1)~{\rm mod}~d,~(j+h-1)~{\rm mod}~d},\] for all $j \in [d]$; similarly we let $N_j''$ denote the respective submatrices of $N''$.  With this notation, it is clear that \[D_j = |\x_0|^{(j)} (|\x_0|^{(j)})^* + N_j'',\] where $|\x_0|^{(j)}_k = |\x_0|_{k + j - 1}, k \in [\delta]$.  This immediately suggests that we can estimate the magnitudes of the entries of $\x_0$ by calculating the top eigenvectors of these approximately rank one $D_j$ matrices.

Indeed, if we do so for all of $D_1, \dots, D_d \in \mathbbm{R}^{\delta \times \delta}$, we will produce $\delta$ estimates of each $(x_0)_{j}$ entry's magnitude.  A final estimate of each $|(x_0)_j|$ can then be computed by taking the average, median, etc. of the $\delta$ different estimates of $|(x_0)_j|$ provided by each of the leading eigenvectors of $D_{j-\delta+1}, \dots, D_j$.  Of course, one need neither use all $d$ possible $D_j$ matrices, nor make them have size $\delta \times \delta$.  More generally, to reduce computational complexity, one may instead use $d/s$ matrices, $\tilde{D}_{j'} \in \mathbbm{R}^{\gamma \times \gamma}$, of size $1 \leq \gamma \leq \delta$ and with shifts $s \leq \gamma$ (dividing $d$), having entries
$$(\tilde{D}_{j'})_{k,h} = |X|_{(sj'+k-1)~{\rm mod}~d,~(sj'+h-1)~{\rm mod}~d}.$$
Computing the leading eigenvectors of $\tilde{D}_{j'}$ for all $j' \in [d/s]$ will then produce (multiple) estimates of each magnitude $|(x_0)_j|$ which can then be averaged, etc., as desired in order produce our final magnitude estimates.  As we shall see below, one can achieve better numerical robustness to noise using this technique than what can be achieved using the simpler magnitude estimation technique presented in line 4 of Algorithm~\ref{alg:phaseRetrieval1}.  

\begin{figure}[hbtp]
\centering
\begin{subfigure}[b]{0.495\textwidth}
\centering
\includegraphics[clip=true, trim = 0.75in 3in 0.75in 3in,
scale=0.5]{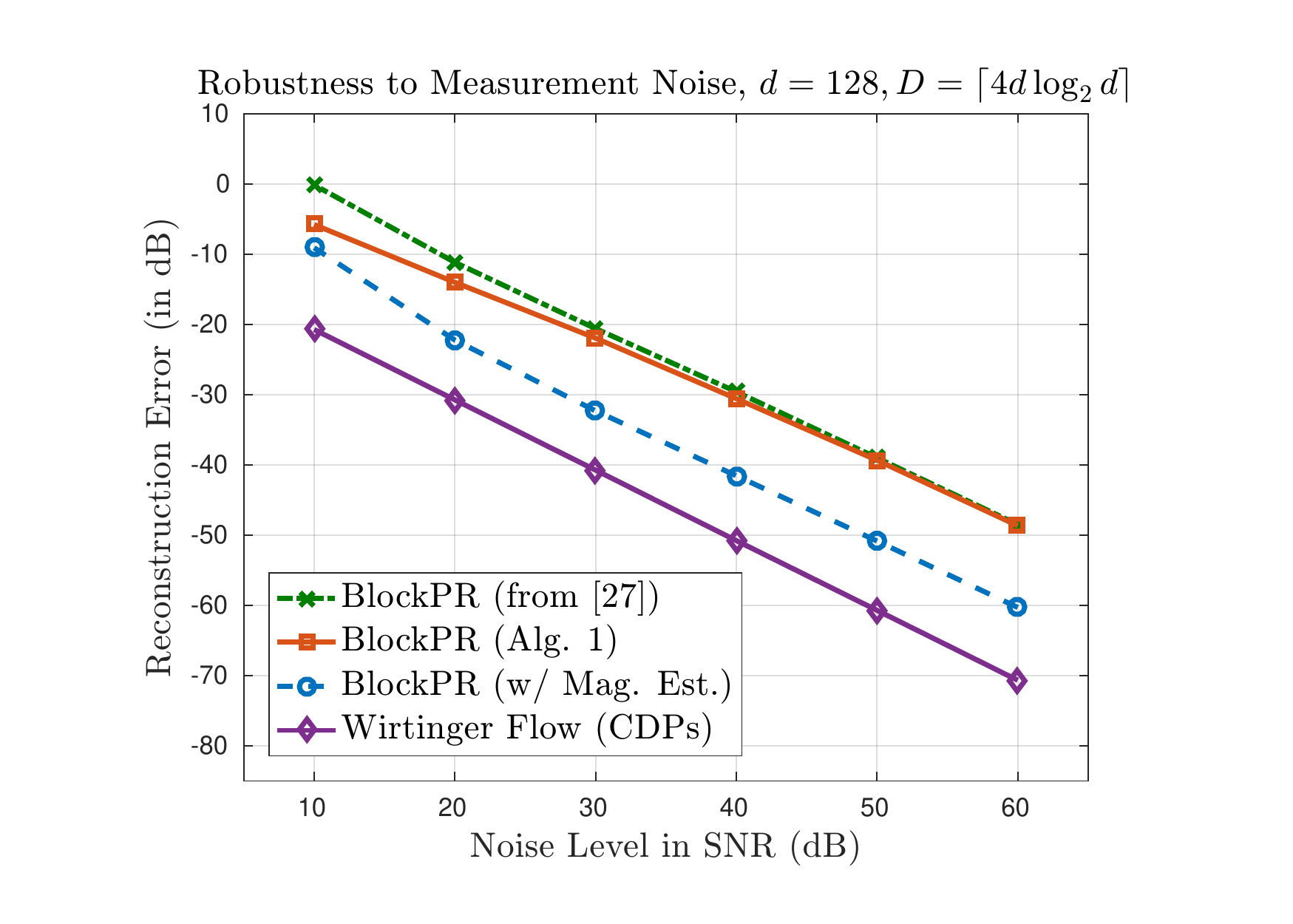}
\caption{Improved Robustness to Measurement Noise -- Comparing Variants of the {\em BlockPR}
algorithm}
\label{fig:eig_vs_greedy}
\end{subfigure}
\hfill
\begin{subfigure}[b]{0.495\textwidth}
\centering
\includegraphics[clip=true, trim = 0.75in 3in 0.75in 3in,
scale=0.5]{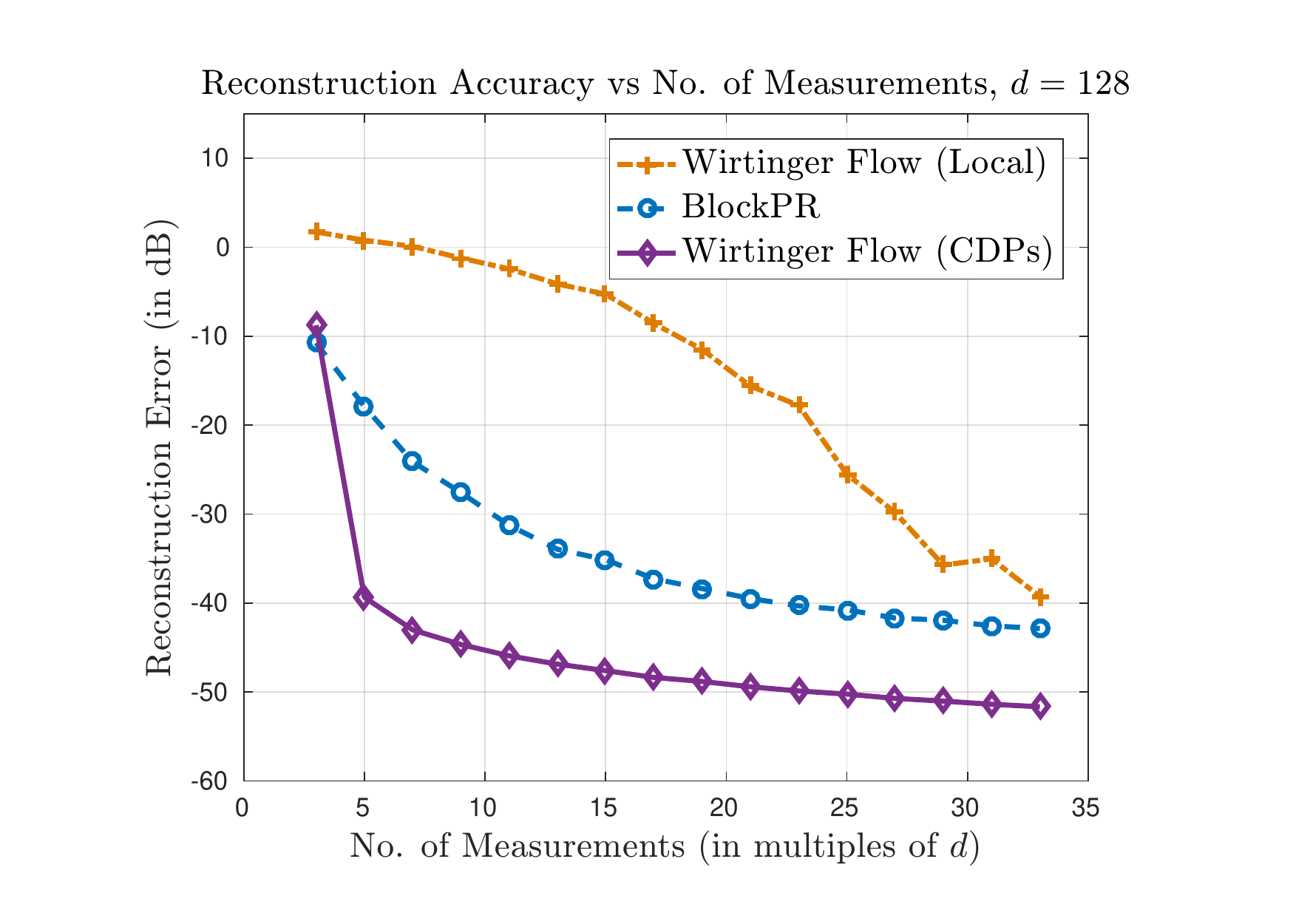}
\caption{Reconstruction Error vs. No. of Measurements; (Reconstruction at $40$dB SNR)}
\label{fig:WF_meas}
\end{subfigure}
\caption{Robust Phase Retrieval -- Local vs. Global Measurements}
\label{fig:WF}
\end{figure}
%

\subsection{Experiments}

We begin by presenting results in Fig. \ref{fig:eig_vs_greedy} demonstrating the improved noise
robustness of the proposed method over the formulation in \cite{IVW2015_FastPhase}. Recall that
\cite{IVW2015_FastPhase} uses a greedy angular synchronization method instead of the
eigenvector-based procedure analyzed in this paper. Fig. \ref{fig:eig_vs_greedy} plots the
reconstruction error when recovering a $d=128$ length complex Gaussian test signal using $D=\lceil
4d\log_2 d\rceil$ measurements at different added noise levels. The local correlation measurements
described in Example 2 of Section \ref{sec:MeasMatrix} are utilized in this plot in and all the ensuing
experiments unless otherwise indicated. Three variants of the proposed algorithm are plotted in 
Fig. \ref{fig:eig_vs_greedy}:
\begin{enumerate}
    \item an implementation of Algorithm \ref{alg:phaseRetrieval1} (denoted by $\square$'s),
    \item an implementation of Algorithm \ref{alg:phaseRetrieval1} with the improved magnitude estimation
        procedure detailed above (with $s=1$ and using the average of the obtained $\tilde D_{j\prime}$
        block magnitude estimates) and post-processed using $100$ iterations of the {\em Gerchberg--Saxton}
        alternating projection algorithm (denoted by $\circ$'s), and 
    \item the algorithmic implementation from \cite{IVW2015_FastPhase} (denoted by $\times$'s). 
\end{enumerate}
We see that the eigenvector-based angular
synchronization method proposed in this paper provides more accurate reconstructions -- especially
at low SNRs -- over the greedy angular synchronization of \cite{IVW2015_FastPhase}. Moreover, the
magnitude estimation procedure detailed above yields significant improvement in reconstruction errors over the two
other variants; consequently, this implementation is used in all plots henceforth. For reference, we
also include reconstruction errors with the {\em Wirtinger Flow} algorithm (denoted by $\Diamond$'s) when using ({\em
global}) coded diffraction pattern (CDP) measurements. Clearly, using global measurements such as
coded diffraction patterns provides superior noise tolerance; however, they are not applicable to
imaging modalities such as ptychography. Indeed, when the {\em Wirtinger Flow} algorithm is used
with local measurements such as those described in this paper, the noise tolerance significantly
deteriorates. Fig. \ref{fig:WF_meas} illustrates this phenomenon by plotting the reconstruction
error in recovering a $d=128$ length complex Gaussian test signal at $40$ dB SNR when using
different numbers of measurements, $D$. {\em Wirtinger flow}, for example, requires a large number
of local measurements before returning accurate reconstructions. 
The wide disparity in
reconstruction accuracy between local and global measurements for {\em Wirtinger Flow} illustrates
the significant challenge in phase retrieval from local measurements. Furthermore, we see that the {\em BlockPR} 
method proposed in this paper is more noise tolerant than {\em Wirtinger Flow} for local
measurements.
\begin{figure}[hbtp]
\centering
\begin{subfigure}[b]{0.495\textwidth}
\centering
\includegraphics[clip=true, trim = 0.75in 3in 0.75in 3in,
scale=0.5]{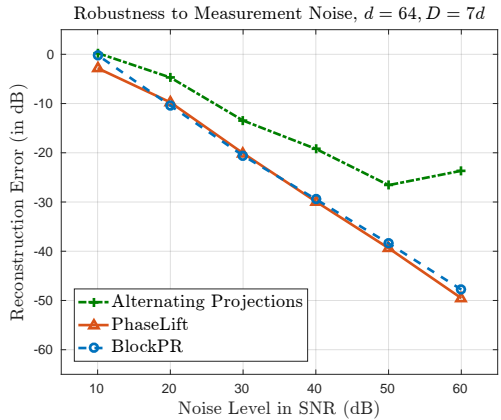}
\caption{Using $D=7d$ measurements.}
\label{fig:noise-7d}
\end{subfigure}
\hfill
\begin{subfigure}[b]{0.495\textwidth}
\centering
\includegraphics[clip=true, trim = 0.75in 3in 0.75in 3in,
scale=0.5]{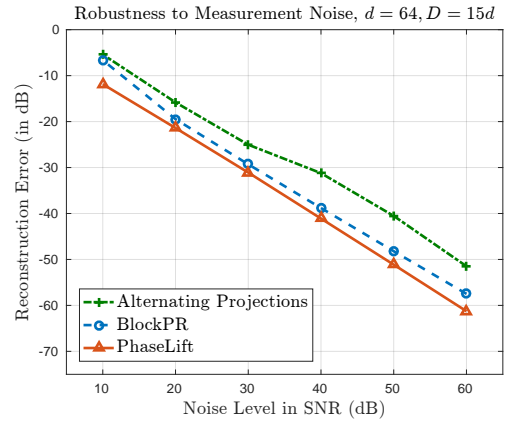}
\caption{Using $D=15d$ measurements.}
\label{fig:noise-15d}
\end{subfigure}
\caption{Robustness to measurement noise -- Phase Retrieval from 
deterministic local correlation measurements.}
\label{fig:noise-local}
\end{figure}
%

Given the weaker performance of {\em Wirtinger Flow} with local measurements, we now restrict our
attention to the empirical evaluation of the proposed method
against the {\em PhaseLift} and {\em Gerchberg-Saxton} alternating projection algorithms. Although
numerical simulations suggest that these methods work with local measurements, we note that (to the
best of our knowledge) there are no theoretical recovery or robustness guarantees for these methods
and measurements.  The {\em PhaseLift} algorithm was implemented as a trace regularized
least-squares problem using CVX \cite{cvx,gb08} -- a package for specifying and solving convex
programs in Matlab; the alternating projection method was initialized with a random complex Gaussian
initial guess and limited to a maximum of $10,000$ iterations.  We begin by presenting numerical
results evaluating the robustness to measurement noise. Figs.
\ref{fig:noise-7d} and \ref{fig:noise-15d} plot the error in reconstructing a $d=64$ length complex
vector $\x_0$ using $D=7d$ and $D=15d$ local correlation-based phaseless measurements respectively.
In particular, the well-conditioned {\em deterministic} measurement construction defined in Example
2 of Section \ref{sec:MeasMatrix} was utilized along with additive Gaussian measurement noise. We see from Fig.
\ref{fig:noise-local} that the method proposed in this paper (denoted {\em BlockPR} in the figure)
performs reliably across a wide range of SNRs and compares favorably against existing popular phase
retrieval algorithms. In particular, the method performs almost as well as the 
{\em PhaseLift} algorithm and returns significantly more accurate reconstructions than the alternating
projections algorithm. We remark that the marginally improved noise robustness of {\em PhaseLift}
is at the expense of a significant increase in computational cost, as we will see in 
Fig. \ref{fig:exectime}.

Next, Fig. \ref{fig:measurements} plots the reconstruction error in recovering a $d=64$-length
complex vector as a function of the number of measurements used. As with Fig.
\ref{fig:noise-local}, the deterministic correlation-based measurement constructions of Section
\ref{sec:MeasMatrix} (Example construction $2$) were utilized along with an additive Gaussian noise model. Plots
are provided for simulations at two noise levels -- $20$ dB and $40$ dB. We observe that the proposed 
algorithm outperforms the popular alternating projections method, and is
almost as accurate as {\em PhaseLift}. Moreover, at the $40$ dB noise level, the proposed method provides the
best reconstruction accuracy when using small numbers of measurements ($D \approx 5d$) which may be
of practical importance. 

Finally, Fig. \ref{fig:exectime} plots the average execution time (in seconds) 
required to solve the phase retrieval problem using $D = \lceil 2d\log_2d\rceil$ noiseless
measurements. For comparison, execution times for the {\em PhaseLift} and alternating projection
algorithms are provided. We observe that the proposed method is several orders of magnitude faster
than the {\em PhaseLift} and alternating projection algorithms. Moreover, the plot confirms the
essentially FFT-time computational complexity (see Section \ref{sec:intro}) of the proposed method.
\begin{figure}[hbtp]
\centering
\begin{subfigure}[b]{0.495\textwidth}
\centering
\includegraphics[clip=true, trim = 0.75in 3in 0.75in 3in,
scale=0.5]{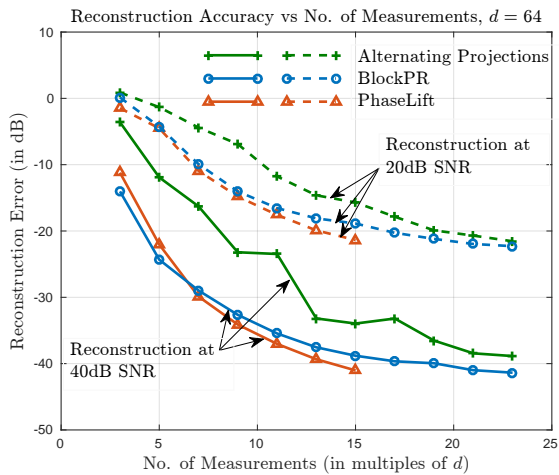}
\caption{Reconstruction Error vs. No. of Measurements}
\label{fig:measurements}
\end{subfigure}
\hfill
\begin{subfigure}[b]{0.495\textwidth}
\centering
\includegraphics[clip=true, trim = 0.75in 3in 0.75in 3in,
scale=0.5]{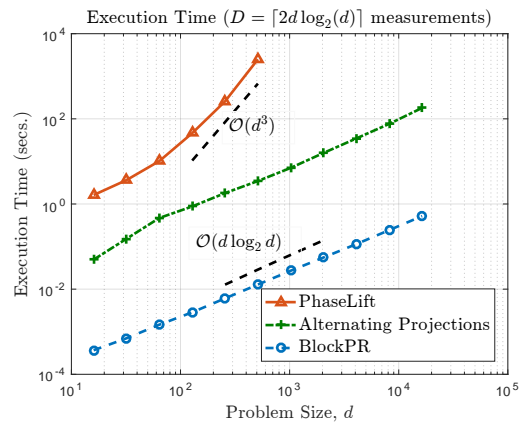}
\caption{Execution Time vs. Problem Size}
\label{fig:exectime}
\end{subfigure}
\caption{Performance Evaluation and Comparison of the Proposed Phase Retrieval Method 
(with Deterministic Local Correlation Measurements and Additive Gaussian Noise)}
\label{fig:performance}
\end{figure}

%% file: Perturb.tex
In this section we present a simpler (and easier to derive), albeit weaker,  perturbation result in the spirit of Section \ref{sec:Perturb}, which is associated with the analysis of line 3 of Algorithm~\ref{alg:phaseRetrieval1}.
Specifically, we will derive an upper bound on $\min_{\theta\in[0,2\pi]} \Vert \tx_0 - \mathbbm{e}^{\mathbbm{i}\theta}\tx \Vert_2$ (provided by Theorem~\ref{cor2:GenBound}), which scales like $d^3$.  While this dependence is strictly worse than the one derived in Section \ref{sec:Perturb}, it is easier to obtain and the technique may be of independent interest. 

We will begin with a result concerning the top eigenvector of any Hermitian matrix.  

\begin{lem}
Let $\X_0 = \sum^d_{j=1} \nu_j \x_j\x_j^*$ be Hermitian with eigenvalues $\nu_1 \geq \nu_2 \geq \dots \geq \nu_d$ and orthonormal eigenvectors $\x_1, \dots, \x_d \in \mathbbm{C}^d$.  Suppose that $\X = \sum^d_{j=1} \lambda_j \v_j\v_j^*$ is Hermitian with eigenvalues $\lambda_1 \geq \lambda_2 \geq \dots \geq \lambda_d$, orthonormal eigenvectors $\v_1, \dots, \v_d \in \mathbbm{C}^d$, and $\Vert \X-\X_0 \Vert_{\rm F}~\leq \eta \Vert \X_0 \Vert_F$ for some $\eta \geq 0$.  Then, 
\[ \left( 1 - | \langle \x_1, \v_1 \rangle |^2 \right)  \leq  \frac{4 \eta^2 \Vert \X_0 \Vert^2_{\rm F} }{(\nu_1- \nu_2)^2} .\]
        \label{lem:GenBound}
\end{lem}

\begin{proof}
An application of the $\sin \theta$ theorem \cite{davis1970rotation,stewart1990matrix} (see, e.g., the proof of Corollary 1 in \cite{yu2015useful}) tells us that 
\begin{equation*}
\sin \left( \arccos \left( | \langle \x_1, \v_1 \rangle | \right) \right) \leq \frac{2 \eta \Vert \X_0 \Vert_F}{|\nu_1 - \nu_2|}.
\label{equ:SinthetaApp}
\end{equation*}
Squaring both sides we then learn that 
\begin{equation}
 \left( 1 - | \langle \x_1, \v_1 \rangle |^2 \right)  = \sin^2 \left( \arccos \left( | \langle \x_1, \v_1 \rangle | \right) \right) \leq \frac{4 \eta^2 \Vert \X_0 \Vert^2_F}{(\nu_1 - \nu_2)^2},
\label{equ:SinthetaApp}
\end{equation}
giving us the desired inequality.  
\end{proof}

The following variant of Lemma~\ref{lem:GenBound} concerning rank 1 matrices $\X_0$ is of use in the analysis of many other phase retrieval methods, and can be used, e.g., to correct and simplify the proof of equation (1.8) in Theorem 1.3 of \cite{candes2014solving}.

\begin{lem}
Let $\x_0 \in \mathbb C^d$, set $\X_0 = \x_0\x_0^*$, and let 
$\X \in \mathbb C^{d \times d}$ be Hermitian with $\Vert \X~-~\X_0 \Vert_F~\leq \eta
\Vert \X_0 \Vert_F = \eta \Vert \x_0\Vert_2^2$ for some $\eta \geq 0$.  Furthermore, let $\lambda_i$ be the $i$-th largest magnitude
eigenvalue of $\X$ and $\v_i \in \mathbb{C}^d$ an associated eigenvector, such 
that the $\v_i$ form an orthonormal eigenbasis. Then
\[ \min_{\theta \in [0, 2 \pi]} \Vert  \mathbbm{e}^{\mathbbm{i} \theta}  \x_0 - \sqrt{ 
    |\lambda_1 |} \v_1 \Vert_2 \leq (1+2 \sqrt{2}) \eta \Vert \x_0 
        \Vert_2.\footnote{It is interesting to note that similar bounds can also be obtained using simpler techniques (see, e.g., \cite{CandesFix}).} \]
        \label{cor:rank1Bound}
\end{lem}

\begin{proof}
In this special case of Lemma~\ref{lem:GenBound} we have $\nu_1 = \Vert \X_0 \Vert_F = \Vert \x_0\Vert_2^2$ and $\x_1 := \x_0 / \Vert \x_0 \Vert$.  Choose $\phi \in [0, 2 \pi]$ such that $\langle \mathbbm{e}^{\mathbbm{i} \phi} \x_0, \v_1 \rangle  = \vert \langle  \x_0, \v_1 \rangle \vert$.  Then, 
\begin{align}
 \Vert \mathbbm{e}^{\mathbbm{i} \phi} \x_0 - \sqrt{\nu_1} \v_1 \Vert_2^2 &= 2\nu_1 - 2 \nu_1  \cdot \vert \langle  \x_0 / \Vert \x_0 \Vert, \v_1 \rangle \vert = 2\nu_1 - 2 \nu_1  \cdot \vert \langle  \x_1, \v_1 \rangle \vert \nonumber \\ &\leq 2 \nu_1 \left( 1 - \vert \langle  \x_1, \v_1 \rangle \vert \right) \left( 1 + \vert \langle  \x_1, \v_1 \rangle \vert \right) \label{equ:CorIntermed} \\ &= 2 \nu_1 \left( 1 - | \langle \x_1, \v_1 \rangle |^2 \right) \leq 8 \eta^2 \Vert \X_0 \Vert_{\rm F} \nonumber
\end{align}
where the last inequality follows from Lemma~\ref{lem:GenBound} with $\nu_1 = \Vert \X_0 \Vert_F = \Vert \x_0\Vert_2^2$.  Finally, by the triangle inequality, Weyl's inequality (see, e.g., \cite{horn2012matrix}), and \eqref{equ:CorIntermed}, we have
\begin{align*}
  \Vert \mathbbm{e}^{\mathbbm{i} \phi} \x_0 - \sqrt{ |\lambda_1|} \v_1 \Vert_2 & \leq \Vert 
    \mathbbm{e}^{\mathbbm{i} \phi} \x_0 -\sqrt \nu_1 \v_1 \Vert_2 + \Vert \sqrt \nu_1 \v_1 - 
    \sqrt{|\lambda_1|} \v_1 \Vert_2 \\
    & \leq 2 \sqrt{2} \cdot \eta \sqrt \nu_1 + \left \vert \sqrt \nu_1 - \sqrt{|\lambda_1|} \right \vert \\ & \leq 2 \sqrt{2} \cdot \eta \sqrt 
        \nu_1 + \frac{| \nu_1 - \lambda_1 |}{\sqrt \nu_1 + \sqrt{|\lambda_1|}} \\
     & \leq 2 \sqrt{2} \cdot \eta \sqrt \nu_1 + \frac{\eta \nu_1}{\sqrt \nu_1 + \sqrt{|\lambda_1|}}\\ 
    & \leq (1+2 \sqrt{2}) \eta \sqrt \nu_1.
\end{align*}
The desired result now follows.
\end{proof}

We may now use Lemma~\ref{lem:GenBound} to produce a perturbation bound for our banded matrix of phase differences $\tilde{\X}_0$ from \eqref{eq:X_0}.

\begin{thm}
Let $\tX_0 = T_{\delta}(\tx_0 \tx_0^*)$ where $|(\tx_0)_i| = 1$ for each $i$.  Further suppose $\tX \in T_{\delta}(\H^d)$ has $\tx$ as its top eigenvector, where $||\tx||_2 = \sqrt{d}$.  Suppose that $\Vert \tX_0 - \tX \Vert_F \le \eta \Vert \tX_0 \Vert_F$ 
for some $\eta>0$.  Then, there exists an absolute constant $C \in \mathbb{R}^+$ such that
\[\min_{\theta\in[0,2\pi]} \Vert \tx_0 - \mathbbm{e}^{\mathbbm{i}\theta} \tx \Vert_2 \le C\bigfrac{\eta d^3}{\delta^{\frac{5}{2}}}.\]

\label{cor2:GenBound}
\end{thm}

\begin{proof}
Recall that the phase vectors $\tilde{\x}$ and $\tilde{\x}_0$ are normalized so that $\| \tilde{\x} \|_2 = \| \tilde{\x}_0 \|_2 = \sqrt{d}$.  Combining Lemmas~\ref{lem:EigGap} and \ref{lem:GenBound} after noting that $\| \tilde{\X}_0 \|^2_{\rm F} = d (2 \delta - 1)$ we learn that 
\begin{equation}
\left( 1 - \frac{1}{d^2}| \langle \tilde{\x}_0, \tilde{\x} \rangle |^2 \right)  \leq  C' \eta^2 \left( \frac{d}{\delta} \right)^5
\label{equ:cor2GenBIP}
\end{equation}
for an absolute constant $C' \in \mathbb{R}^+$.  Let $\phi \in [0,2 \pi)$ be such that $\operatorname{Re}\left( \langle \tilde{\x}_0,\mathbbm{e}^{\mathbbm{i} \phi} \tilde{\x} \rangle \right) = \left| \langle \tilde{\x}_0, \tilde{\x} \rangle \right|$.  Then,
\begin{align*}
\| \tilde{\x}_0 - \mathbbm{e}^{\mathbbm{i} \phi} \tilde{\x} \|^2_2 &= 2d - 2 \operatorname{Re}\left( \langle \tilde{\x}_0,\mathbbm{e}^{\mathbbm{i} \phi} \tilde{\x} \rangle \right) \\
&= 2d \left( 1 - \frac{1}{d} \left| \langle \tilde{\x}_0, \tilde{\x} \rangle \right| \right) \leq 2d \left( 1 - \frac{1}{d^2} \left| \langle \tilde{\x}_0, \tilde{\x} \rangle \right|^2 \right).
\end{align*}
Combining this last inequality with \eqref{equ:cor2GenBIP} concludes the proof.
\end{proof}

%% file: BPRoughDraft_V2.bbl
\begin{thebibliography}{10}

\bibitem{alexeev2014phase}
B.~Alexeev, A.~S. Bandeira, M.~Fickus, and D.~G. Mixon.
\newblock Phase retrieval with polarization.
\newblock {\em SIAM Journal on Imaging Sciences}, 7(1):35--66, 2014.

\bibitem{balan2009painless}
R.~Balan, B.~G. Bodmann, P.~G. Casazza, and D.~Edidin.
\newblock Painless reconstruction from magnitudes of frame coefficients.
\newblock {\em Journal of Fourier Analysis and Applications}, 15(4):488--501,
  2009.

\bibitem{balan2006signal}
R.~Balan, P.~Casazza, and D.~Edidin.
\newblock On signal reconstruction without phase.
\newblock {\em Applied and Computational Harmonic Analysis}, 20(3):345--356,
  2006.

\bibitem{candes2014solving}
E.~J. Candes and X.~Li.
\newblock Solving quadratic equations via phaselift when there are about as
  many equations as unknowns.
\newblock {\em Foundations of Computational Mathematics}, 14(5):1017--1026,
  2014.

\bibitem{Candes2014WF}
E.~J. Candes, X.~Li, and M.~Soltanolkotabi.
\newblock Phase retrieval from coded diffraction patterns.
\newblock {\em Applied and Computational Harmonic Analysis}, 39(2):277--299,
  Sept. 2015.

\bibitem{candes2015phase}
E.~J. Candes, X.~Li, and M.~Soltanolkotabi.
\newblock Phase retrieval via wirtinger flow: Theory and algorithms.
\newblock {\em Information Theory, IEEE Transactions on}, 61(4):1985--2007,
  2015.

\bibitem{candes2013phaselift}
E.~J. Candes, T.~Strohmer, and V.~Voroninski.
\newblock Phaselift: Exact and stable signal recovery from magnitude
  measurements via convex programming.
\newblock {\em Communications on Pure and Applied Mathematics},
  66(8):1241--1274, 2013.

\bibitem{davis1970rotation}
C.~Davis and W.~M. Kahan.
\newblock The rotation of eigenvectors by a perturbation. {III}.
\newblock {\em SIAM Journal on Numerical Analysis}, 7(1):1--46, 1970.

\bibitem{eldar2014sparse}
Y.~Eldar, P.~Sidorenko, D.~Mixon, S.~Barel, and O.~Cohen.
\newblock Sparse phase retrieval from short-time fourier measurements.
\newblock {\em IEEE Signal Proc. Letters}, 22(5), 2015.

\bibitem{fienup1978reconstruction}
J.~R. Fienup.
\newblock Reconstruction of an object from the modulus of its fourier
  transform.
\newblock {\em Optics letters}, 3(1):27--29, 1978.

\bibitem{gerchberg1972practical}
R.~Gerchberg and W.~Saxton.
\newblock A practical algorithm for the determination of the phase from image
  and diffraction plane pictures.
\newblock {\em Optik}, 35:237—246, 1972.

\bibitem{gross2015improved}
D.~Gross, F.~Krahmer, and R.~Kueng.
\newblock Improved recovery guarantees for phase retrieval from coded
  diffraction patterns.
\newblock {\em Applied and Computational Harmonic Analysis}, 2015.

\bibitem{horn2012matrix}
R.~A. Horn and C.~R. Johnson.
\newblock {\em Matrix Analysis}.
\newblock Cambridge University Press, 2 edition, 2012.

\bibitem{CandesFix}
M.~Iwen, F.~Krahmer, and A.~Viswanathan.
\newblock Technical note: A minor correction of theorem 1.3 from [1].
\newblock {\em Unpublished note available at
  \url{http://users.math.msu.edu/users/markiwen/Papers/PhaseLiftproof.pdf}},
  April 2015.

\bibitem{IVW2015_FastPhase}
M.~Iwen, A.~Viswanathan, and Y.~Wang.
\newblock Fast phase retrieval from local correlation measurements.
\newblock {\em Submitted}, 2015.

\bibitem{bitbucket_BlockPR}
M.~Iwen, Y.~Wang, and A.~Viswanathan.
\newblock {BlockPR}: Matlab software for phase retrieval using block circulant
  measurement constructions and angular synchronization, version 2.0.
\newblock \url{https://bitbucket.org/charms/blockpr}, 2016.

\bibitem{jaganathan2015stft}
K.~Jaganathan, Y.~C. Eldar, and B.~Hassibi.
\newblock Stft phase retrieval: Uniqueness guarantees and recovery algorithms.
\newblock {\em arXiv preprint arXiv:1508.02820}, 2015.

\bibitem{marchesini2015alternating}
S.~Marchesini, Y.-C. Tu, and H.-t. Wu.
\newblock Alternating projection, ptychographic imaging and phase
  synchronization.
\newblock {\em Applied and Computational Harmonic Analysis}, 2015.

\bibitem{netrapalli2013phase}
P.~Netrapalli, P.~Jain, and S.~Sanghavi.
\newblock Phase retrieval using alternating minimization.
\newblock In {\em Advances in Neural Information Processing Systems}, pages
  2796--2804, 2013.

\bibitem{Rodenburg2008}
J.~Rodenburg.
\newblock Ptychography and related diffractive imaging methods.
\newblock {\em Advances in Imaging and Electron Physics}, 150:87--184, 2008.

\bibitem{salanevich2015polarization}
P.~Salanevich and G.~E. Pfander.
\newblock Polarization based phase retrieval for time-frequency structured
  measurements.
\newblock In {\em Sampling Theory and Applications (SampTA), 2015 International
  Conference on}, pages 187--191. IEEE, 2015.

\bibitem{shechtman2015phase}
Y.~Shechtman, Y.~C. Eldar, O.~Cohen, H.~N. Chapman, J.~Miao, and M.~Segev.
\newblock Phase retrieval with application to optical imaging: a contemporary
  overview.
\newblock {\em Signal Processing Magazine, IEEE}, 32(3):87--109, 2015.

\bibitem{stewart1990matrix}
G.~Stewart and J.~Sun.
\newblock {\em Matrix Perturbation Theory}.
\newblock Academic Press, 1990.

\bibitem{IV_SPIE}
A.~Viswanathan and M.~Iwen.
\newblock Fast angular synchronization for phase retrieval via incomplete
  information.
\newblock {\em Proc. of SPIE Optics + Photonics}, 2015.

\bibitem{yu2015useful}
Y.~Yu, T.~Wang, and R.~Samworth.
\newblock A useful variant of the davis--kahan theorem for statisticians.
\newblock {\em Biometrika}, 102(2):315--323, 2015.

\end{thebibliography}


\begin{thebibliography}{10}

\bibitem{alexeev2014phase}
B.~Alexeev, A.~S. Bandeira, M.~Fickus, and D.~G. Mixon.
\newblock Phase retrieval with polarization.
\newblock {\em SIAM Journal on Imaging Sciences}, 7(1):35--66, 2014.

\bibitem{balan2007fast}
R.~Balan, B.~Bodmann, P.~Casazza, and D.~Edidin.
\newblock Fast algorithms for signal reconstruction without phase.
\newblock In {\em Optical Engineering+ Applications}, pages 67011L--67011L.
  International Society for Optics and Photonics, 2007.

\bibitem{balan2009painless}
R.~Balan, B.~G. Bodmann, P.~G. Casazza, and D.~Edidin.
\newblock Painless reconstruction from magnitudes of frame coefficients.
\newblock {\em Journal of Fourier Analysis and Applications}, 15(4):488--501,
  2009.

\bibitem{balan2006signal}
R.~Balan, P.~Casazza, and D.~Edidin.
\newblock On signal reconstruction without phase.
\newblock {\em Applied and Computational Harmonic Analysis}, 20(3):345--356,
  2006.

\bibitem{bandeira2013near}
A.~Bandeira and D.~Mixon.
\newblock Near-optimal phase retrieval of sparse vectors.
\newblock {\em arXiv preprint arXiv:1308.0143}, 2013.

\bibitem{Cheeger}
A.~S. {Bandeira}, A.~{Singer}, and D.~A. {Spielman}.
\newblock {A Cheeger Inequality for the Graph Connection Laplacian}.
\newblock {\em ArXiv e-prints}, Apr. 2012.

\bibitem{BendoryE16}
T.~Bendory and Y.~C. Eldar.
\newblock Non-convex phase retrieval from {STFT} measurements.
\newblock {\em CoRR}, abs/1607.08218, 2016.

\bibitem{bodmann2013stable}
B.~G. {Bodmann} and N.~{Hammen}.
\newblock {Stable phase retrieval with low-redundancy frames}.
\newblock {\em ArXiv e-prints}, Feb. 2013.

\bibitem{candes2012phaselift}
E.~Cand{\`e}s, T.~Strohmer, and V.~Voroninski.
\newblock Phaselift: Exact and stable signal recovery from magnitude
  measurements via convex programming.
\newblock {\em Communications on Pure and Applied Mathematics}, 2012.

\bibitem{candes2014solving}
E.~J. Candes and X.~Li.
\newblock Solving quadratic equations via phaselift when there are about as
  many equations as unknowns.
\newblock {\em Foundations of Computational Mathematics}, 14(5):1017--1026,
  2014.

\bibitem{Candes2014WF}
E.~J. Candes, X.~Li, and M.~Soltanolkotabi.
\newblock Phase retrieval from coded diffraction patterns.
\newblock {\em Applied and Computational Harmonic Analysis}, 39(2):277--299,
  Sept. 2015.

\bibitem{candes2015phase}
E.~J. Candes, X.~Li, and M.~Soltanolkotabi.
\newblock Phase retrieval via wirtinger flow: Theory and algorithms.
\newblock {\em Information Theory, IEEE Transactions on}, 61(4):1985--2007,
  2015.

\bibitem{candes2013phaselift}
E.~J. Candes, T.~Strohmer, and V.~Voroninski.
\newblock Phaselift: Exact and stable signal recovery from magnitude
  measurements via convex programming.
\newblock {\em Communications on Pure and Applied Mathematics},
  66(8):1241--1274, 2013.

\bibitem{chungspectral}
F.~Chung.
\newblock {\em Spectral Graph Theory}.
\newblock Number no. 92 in CBMS Regional Conference Series. Conference Board of
  the Mathematical Sciences, 1992.

\bibitem{davis1970rotation}
C.~Davis and W.~M. Kahan.
\newblock The rotation of eigenvectors by a perturbation. {III}.
\newblock {\em SIAM Journal on Numerical Analysis}, 7(1):1--46, 1970.

\bibitem{Dierolf2008Ptych}
M.~Dierolf, O.~Bunk, S.~Kynde, P.~Thibault, I.~Johnson, A.~Menzel, K.~Jefimovs,
  C.~David, O.~Marti, and F.~Pfeiffer.
\newblock Ptychography \& lensless {X}-ray imaging.
\newblock {\em Europhysics News}, 39(1):22--24, 2008.

\bibitem{eldar2014sparse}
Y.~Eldar, P.~Sidorenko, D.~Mixon, S.~Barel, and O.~Cohen.
\newblock Sparse phase retrieval from short-time fourier measurements.
\newblock {\em IEEE Signal Proc. Letters}, 22(5), 2015.

\bibitem{eldar2012phase}
Y.~C. Eldar and S.~Mendelson.
\newblock Phase retrieval: Stability and recovery guarantees.
\newblock {\em Applied and Computational Harmonic Analysis}, 36(3):473--494,
  2014.

\bibitem{fienup1978reconstruction}
J.~R. Fienup.
\newblock Reconstruction of an object from the modulus of its fourier
  transform.
\newblock {\em Optics letters}, 3(1):27--29, 1978.

\bibitem{gerchberg1972practical}
R.~Gerchberg and W.~Saxton.
\newblock A practical algorithm for the determination of the phase from image
  and diffraction plane pictures.
\newblock {\em Optik}, 35:237—246, 1972.

\bibitem{Goodman2005IntroFourierOptics}
J.~W. Goodman.
\newblock {\em Introduction to {F}ourier optics}.
\newblock Roberts and Company Publishers, 2005.

\bibitem{gb08}
M.~Grant and S.~Boyd.
\newblock Graph implementations for nonsmooth convex programs.
\newblock In V.~Blondel, S.~Boyd, and H.~Kimura, editors, {\em Recent Advances
  in Learning and Control}, Lecture Notes in Control and Information Sciences,
  pages 95--110. Springer-Verlag, 2008.
\newblock \url{http://stanford.edu/~boyd/graph_dcp.html}.

\bibitem{cvx}
M.~Grant and S.~Boyd.
\newblock {CVX}: Matlab software for disciplined convex programming, version
  2.1.
\newblock \url{http://cvxr.com/cvx}, Mar. 2014.

\bibitem{gross2015improved}
D.~Gross, F.~Krahmer, and R.~Kueng.
\newblock Improved recovery guarantees for phase retrieval from coded
  diffraction patterns.
\newblock {\em Applied and Computational Harmonic Analysis}, 2015.

\bibitem{horn2012matrix}
R.~A. Horn and C.~R. Johnson.
\newblock {\em Matrix Analysis}.
\newblock Cambridge University Press, 2 edition, 2012.

\bibitem{CandesFix}
M.~Iwen, F.~Krahmer, and A.~Viswanathan.
\newblock Technical note: A minor correction of theorem 1.3 from [1].
\newblock {\em Unpublished note available at
  \url{http://users.math.msu.edu/users/markiwen/Papers/PhaseLiftproof.pdf}},
  April 2015.

\bibitem{IVW2015_FastPhase}
M.~Iwen, A.~Viswanathan, and Y.~Wang.
\newblock Fast phase retrieval from local correlation measurements.
\newblock {\em SIAM Journal on Imaging Sciences}, 9(4):1655--1688, 2016.

\bibitem{bitbucket_BlockPR}
M.~Iwen, Y.~Wang, and A.~Viswanathan.
\newblock {BlockPR}: Matlab software for phase retrieval using block circulant
  measurement constructions and angular synchronization, version 2.0.
\newblock \url{https://bitbucket.org/charms/blockpr}, Apr. 2016.

\bibitem{jaganathan2015stft}
K.~Jaganathan, Y.~C. Eldar, and B.~Hassibi.
\newblock Stft phase retrieval: Uniqueness guarantees and recovery algorithms.
\newblock {\em arXiv preprint arXiv:1508.02820}, 2015.

\bibitem{li2012sparse}
X.~Li and V.~Voroninski.
\newblock Sparse signal recovery from quadratic measurements via convex
  programming.
\newblock {\em SIAM Journal on Mathematical Analysis}, 45(5):3019--3033, 2013.

\bibitem{marchesini2015alternating}
S.~Marchesini, Y.-C. Tu, and H.-t. Wu.
\newblock Alternating projection, ptychographic imaging and phase
  synchronization.
\newblock {\em Applied and Computational Harmonic Analysis}, 2015.

\bibitem{millane1990phase}
R.~Millane.
\newblock Phase retrieval in crystallography and optics.
\newblock {\em J. Opt. Soc. Am. A}, 7(3):394--411, 1990.

\bibitem{netrapalli2013phase}
P.~Netrapalli, P.~Jain, and S.~Sanghavi.
\newblock Phase retrieval using alternating minimization.
\newblock In {\em Advances in Neural Information Processing Systems}, pages
  2796--2804, 2013.

\bibitem{pfander2016robust}
G.~E. Pfander and P.~Salanevich.
\newblock Robust phase retrieval algorithm for time-frequency structured
  measurements.
\newblock {\em eprint arXiv:1611.02540}, 2016.

\bibitem{salanevich2015polarization}
P.~Salanevich and G.~E. Pfander.
\newblock Polarization based phase retrieval for time-frequency structured
  measurements.
\newblock In {\em Sampling Theory and Applications (SampTA), 2015 International
  Conference on}, pages 187--191. IEEE, 2015.

\bibitem{stewart1990matrix}
G.~Stewart and J.~Sun.
\newblock {\em Matrix Perturbation Theory}.
\newblock Academic Press, 1990.

\bibitem{trefethen1997numerical}
L.~N. Trefethen and D.~Bau~III.
\newblock {\em Numerical linear algebra}, volume~50.
\newblock Siam, 1997.

\bibitem{IV_SPIE}
A.~Viswanathan and M.~Iwen.
\newblock Fast angular synchronization for phase retrieval via incomplete
  information.
\newblock {\em Proc. of SPIE Optics + Photonics}, 2015.

\bibitem{walther1963question}
A.~{Walther}.
\newblock {The Question of Phase Retrieval in Optics}.
\newblock {\em Optica Acta}, 10:41--49, 1963.

\bibitem{arXiv:1105.5628_MarchesiniEtAl}
C.~{Yang}, J.~{Qian}, A.~{Schirotzek}, F.~{Maia}, and S.~{Marchesini}.
\newblock {Iterative Algorithms for Ptychographic Phase Retrieval}.
\newblock {\em ArXiv e-prints}, May 2011.

\bibitem{yu2015useful}
Y.~Yu, T.~Wang, and R.~Samworth.
\newblock A useful variant of the davis--kahan theorem for statisticians.
\newblock {\em Biometrika}, 102(2):315--323, 2015.

\end{thebibliography}
